\documentclass[12pt]{amsart}
\usepackage{amsmath,amsthm,amssymb}
\usepackage{enumerate}
\usepackage{graphicx}
\usepackage{epstopdf}
\usepackage{cite}
\usepackage[margin=1in]{geometry}
\usepackage{tikz}
\usetikzlibrary{matrix}
\usetikzlibrary{arrows}
\usepackage{overpic}

\theoremstyle{plain}
\newtheorem{thm}{Theorem}[section]
\newtheorem{cor}[thm]{Corollary}
\newtheorem{lem}[thm]{Lemma}

\theoremstyle{definition}
\newtheorem{defn}[thm]{Definition}
\theoremstyle{remark}
\newtheorem{rem}[thm]{Remark}

\newcommand{\R}{\mathbb{R}}

\newcommand{\Z}{\mathbb{Z}}
\newcommand{\Q}{\mathbb{Q}}

\begin{document}

\title{The Knot Floer Cube of Resolutions and the Composition Product}

\author{Nathan Dowlin}
\date{}

\maketitle

\begin{abstract}

We examine the relationship between the (untwisted) knot Floer cube of resolutions and HOMFLY-PT homology. By using a filtration induced by additional basepoints on the Heegaard diagram for a knot $K$, we see that the filtered complex decomposes as a direct sum of HOMFLY-PT homologies of various subdiagrams. Jaeger's composition product formula shows that the graded Euler characteristic of this direct sum is the HOMFLY-PT polynomial of $K$.

\end{abstract}

\section{Introduction}

In \cite{Gukov}, Dunfield, Gukov, and Rasmussen conjectured a framework for unifying the $sl(N)$ homologies (for all N) of Khovanov and Rozansky and knot Floer homology of Ozsv\'{a}th and Szab\'{o}. Rasmussen successfully unified the $sl(N)$ homologies in \cite{Rasmussen} by finding a class of spectral sequences $E_{k}(N), N \ge 1$ starting at HOMFLY-PT homology and converging to $sl(N)$ homology, and even found a spectral sequence $E_{k}(-1)$ converging to  `$sl(-1)$' homology, which was evidence for a conjectured symmetry on HOMFLY-PT homology, but the relationship between Khovanov Rozansky homology and knot Floer homology remained a mystery.

A key property of HOMFLY-PT homology and $sl(N)$ homology that allowed this relationship was their ability to be constructed via an oriented cube of resolutions. The first such construction for knot Floer homology was given by Ozsv\'{a}th and Szab\'{o} in \cite{Szabo} with twisted coefficients, but they noted a similarity between the specialization to $t=1$ and HOMFLY-PT homology. This similarity was further studied by Manolescu in \cite{Manolescu} and by Gilmore in \cite{Gilmore}. Manolescu conjectured that if $C_{F}(D)$ is the cube of resolutions complex and $d_{i}$ is the component of the differential which increases the cube grading by $i$, then 

\[ H_{*}(H_{*}(C_{F}(D),d_{0}), d_{1}^{*}) \cong H_{H}(K) \]

\noindent
where $D$ is a braid diagram for a knot $K$ and $H_{H}(K)$ is the HOMFLY-PT homology of $K$.

We find that by putting a filtration on the knot Floer cube of resolutions, we can show that Manolescu's conjecture is very closely related to a categorification of Jaeger's composition product \cite{Jaeger}, and that if we filter the differentials $d_{0}$ and $d_{1}^{*}$ by adding a basepoint to the Heegaard diagram, then the graded Euler characteristic of the resulting homology is the HOMFLY-PT polynomial. The relationship with Jaeger's composition product is unexpected, as the composition product is combinatorial by nature and has historically been more associated with the quantum $sl_n$ and HOMFLY-PT invariants. 

We will describe the composition product with the HOMFLY-PT polynomial $P_{H}(a,q,L)$ defined in terms of the skein relation 

\[ aP_{H}(a,q,L_{+}) -a^{-1}P_{H}(a,q,L_{-})= (q-q^{-1})P_{H}(a,q,L_{0}) \]

\noindent
where $L_{+}$, $L_{-}$, and $L_{0}$ are identical except at one crossing, where $L_{+}$ has a positive crossing, $L_{-}$ has a negative crossing, and $L_{0}$ has the oriented smoothing. The invariant is uniquely determined by this relation, together with the normalization $P_{H}(unknot)=1$. We define the single-variable polynomial $P_{n}(q,L)$ by 

\[ P_{n}(q,L) = P_{H}(q^{n},q,L) \]

For $n \ge 1$, $P_{n}(q,L)$ gives the $sl_n$ polynomial of $L$, and $P_{0}(q,L)$ is the Alexander polynomial. The specialization of the HOMFLY-PT polynomial to the $sl_n$ polynomial corresponds to Rasmussen's spectral sequences from HOMFLY-PT homology to $sl_n$ homology, and the specialization to the Alexander polynomial gives motivation for why we might expect a spectral sequence to knot Floer homology as well. 

In order to talk about Jaeger's composition product, we must first define labelings of a diagram. Let $K$ be a knot with corresponding diagram $D$. Viewing $D$ as an oriented 4-valent graph, we say that a subset $S$ of the edges of $D$ is a cycle if at each vertex in $D$ the number of incoming edges in $S$ is equal the the number of outgoing edges in $S$. A \emph{labeling} $f$ of the diagram $D$ is a function from the set of edges in $D$ to the set $\{1,2\}$ such that $f^{-1}(1)$ is a cycle. (Note that  $f^{-1}(1)$ is a cycle iff $f^{-1}(2)$ is a cycle.)

We will put two restrictions on which cycles are allowed. First, we will make $D$ a decorated diagram, i.e. we will choose a marked edge $e_{0}$, and we will require that $f(e_0)=2$. Second, a cycle is said to make a turn at a crossing $c$ if the cycle has one incoming edge at $c$ and one outgoing edge at $c$, and those edges are not diagonal from one another. A labeling $f$ is \emph{admissible} if the cycle $f^{-1}(1)$ doesn't make any `left turns' at positive crossings or `right turns' at negative crossings. The non-admissible labelings are shown in Figure \ref{nonadmissible}.

\begin{figure}[h!]
 \centering
   \begin{overpic}[width=.7\textwidth]{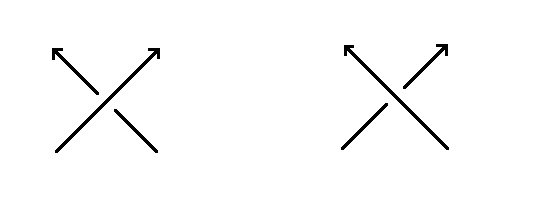}
   \put(7,9){$1$}
   \put(7,25){$1$}
   \put(29.5,9){$2$}
   \put(29.5,25){$2$}
   
   \put(59,9){$2$}
   \put(59.5,25){$2$}
   \put(82,9){$1$}
   \put(82,25){$1$}
   \end{overpic}
\caption{Non-Admissible Labelings} \label{nonadmissible}
\end{figure}

Since the cycle $f^{-1}(1)$ uniquely determines the labeling $f$, we will say that a cycle $Z$ is admissible if the unique labeling $f$ with $f^{-1}(1)=Z$ is admissible. The two cycles $f^{-1}(1)$ and $f^{-1}(2)$ can both be viewed as diagrams of links if we retain the crossing information whenever one of them contains all four edges at a crossing, and forget it otherwise. We will refer to these diagrams as $D_{f,1}$ and $D_{f,2}$, respectively. 

Define $s(D_{f,i})$ to be the sum of the signs of the crossings in $D$ such that $f^{-1}(i)$ has at least one edge incident to the crossing. (Note: This is a slight abuse of notation, as $s(D_{f,i})$ depends on $D$ as well as $D_{f,i}$.) With this language, the $m=1$ version of the composition product can be stated as

\begin{equation}\label{p1}
\begin{split}
  \mathop{\sum_{f \text{ admissible}}}_{f(e_{0})=2}  \Big[ (-1)^{\mathcal{T}_{-}(f)}(q-q^{-1})^{\mathcal{T}(f)} & q^{r(D_{f,2})-s(D_{f,2})}a^{-r(D_{f,1})-s(D_{f,1})} \\
 & \cdot P_{H}(q,q,D_{f,1}) P_{H}(a,q,D_{f,2}) \Big]   = P_{H}(aq,q,D)
\end{split}  
\end{equation}

\noindent
where $\mathcal{T}_{-}(f)$ is the number of turns at negative crossings in the cycle $f^{-1}(1)$, $\mathcal{T}(f)$ is the total number of turns, and $r$ is the rotation number of a cycle, which will be defined in Section 2. 

A reader familiar with the composition product will note that this is not quite the traditional definition. Aside from the superficial change of not shifting $P_{H}$ by the writhe of $D$, the standard definition has no marked edge, and comes with a normalization factor of $\frac{a-a^{-1}}{q-q^{-1}}$. The reason for this is that we are trying to relate the middle HOMFLY-PT homology to knot Floer homology, but the standard composition product would categorify to the unreduced HOMFLY-PT homology, which is twice as large. These differences will be discussed in Section \ref{2.1}, and this unreduced version can be achieved by simply adding an unlinked component, and placing the marked edge on it. 

We will also be interested in the specialization to $a=q^{-1}$, which gives the Alexander polynomial.

\begin{equation}
\label{p2}
 \mathop{\sum_{f \text{ admissible}}}_{f(e_{0})=2} (-1)^{\mathcal{T}_{-}(f)}(q-q^{-1})^{\mathcal{T}(f)}q^{r(D)+s(D_{f,1})-s(D_{f,2})} P_{1}(q,D_{f,1})P_{-1}(q,D_{f,2}) = P_{0}(q,D)
 \end{equation}

We will place two filtrations on the knot Floer cube of resolutions, one coming from height in the cube, and the other coming from extra basepoints in the Heegaard diagram, to be described in Section \ref{basepointsection}. The resulting complex will have a triple grading: the Maslov grading, the Alexander grading, and the cube grading. In all cases we will work over $\Z_{2}$, as signs in knot Floer homology can be difficult to compute. However, HOMFLY-PT homology is a well-defined invariant over $\Z_{2}$ as well, so this will not cause any problems.

Since HOMFLY-PT homology has only been proved to be invariant under braid-like Reidemeister moves, we will assume $D$ is in decorated braid position - note that this implies that $D_{f,i}$ is also a braid for any labeling $f$. 

\begin{thm}
\label{theorem1}

Let $C_{F}(D)$ be the complex given by the (untwisted) cube of resolutions for knot Floer homology, and $d^{f}_{i}$ the differentials that change the cube filtration by $i$ and preserve the basepoint filtration. Then, ignoring the grading shift for each $f$,

\begin{equation}
\label{bigthm}
H_{*}(H_{*}(C_{F}(D),d^{f}_{0}),(d^{f}_{1})^{*})=  \mathop{\bigoplus_{f \text{ admissible}}}_{f(e_{0})=2}  H_{1}(D_{f,1}) \otimes H_{H} (D_{f,2},\mathcal{T}(f)) 
\end{equation}

\noindent
where $H_{H}(D,k)$ is the HOMFLY-PT homology of $D$ reduced $k$ times. 

In particular, there is a spectral sequence from the right side of (\ref{bigthm}) to $\mathit{HFK}^{-}(K)$. For a spectral sequence converging to $\widehat{\mathit{HFK}}(K)$, we change $\mathcal{T}(Z)$ to $\mathcal{T}(Z)+1$.

\end{thm}

We define reducing at an edge $e_{i}$ to be tensoring with the complex $R \xrightarrow{U_{i}} R$. With the proper gradings, this complex will contribute a $q-q^{-1}$ to the graded Euler characteristic. Thus, these reductions give us the $(q-q^{-1})^{\mathcal{T}(f)}$ in the composition product. The $\mathcal{T}(Z)$ reductions of $H_{H}(D_{f,2})$ will be done so that each component gets reduced at least once, which makes this expression both well-defined and, in the reduced case, finite-dimensional. Therefore, since $H_1$ categorifies the $sl_1$ polynomial and $H_{H}$ categorifies the HOMFLY-PT polynomial, applying (\ref{p1}) we get the following corollary:

\begin{cor}

The graded Euler characteristic of the basepoint - filtered $E_{2}$ page 

\[H_{*}(H_{*}(C_{F}(D),d^{f}_{0}),(d^{f}_{1})^{*})\]

\noindent
is the HOMFLY-PT polynomial $P_{H}(aq,q,D)$.

\end{cor}

By relaxing the cube filtration, we have instead a doubly graded complex. These complexes are related via Rasmussen's $E_{k}(-1)$ spectral sequence - we start at the (1, HOMFLY-PT) composition product, and by running the $E_{k}(-1)$ spectral sequence on HOMFLY-PT homology we get a formulation of knot Floer homology in terms of the (1, -1) composition product.

\begin{thm}

With $d^{f}_{i}$ defined as above, and again ignoring the grading shifts,

\begin{equation}
\label{thm2}
H_{*}(C_{F}(D),d^{f}_{0}+d^{f}_{1})=  \mathop{\bigoplus_{f \text{ admissible}}}_{f(e_{0})=2}  H_{1}(D_{f,1}) \otimes H_{-1} (D_{f,2},\mathcal{T}(f))
\end{equation}

\noindent
or, in the case of reduced knot Floer homology,

\begin{equation}
\label{thm2red}
 H_{*}(\overline{C}_{F}(K),d_{00}+d_{10})=  \mathop{\bigoplus_{f \text{ admissible}}}_{f(e_{0})=2}  H_{1}(D_{f,1}) \otimes H_{-1} (D_{f,2},\mathcal{T}(f)+1)=  \mathop{\bigoplus_{f \text{ admissible}}}_{f(e_{0})=2} V^{\otimes \mathcal{T}(f)} 
 \end{equation}

\noindent
where $V$ is a two-dimensional vector space over $\Z_{2}$. There are differentials on these complexes giving $HFK^{-}(K)$ and $\widehat{HFK}(K)$, respectively.

\end{thm}

The simplification in (\ref{thm2red}) follows from the simplicity of $H_{1}$ and $H_{-1}$. For any link $L$, $H_{1}(L)$ is always one-dimensional. $H_{-1}$ is slightly more complicated, because if any component of $L$ is not reduced, then $H_{-1}$ is infinite dimensional. However, in the case that every component is reduced on at least one edge, we have that $H_{-1}(L,k)=V^{\otimes k-1}$.

Theorem \ref{theorem1} computes the same quantity conjectured by Manolescu to give HOMFLY-PT homology, except that the differentials have been filtered by a large set of basepoints in the Heegaard diagram. These basepoints allow the homology to be computed quite easily. However, they come at the cost of our homology being much larger than HOMFLY-PT homology, and clearly not invariant under Reidemeister moves. We hope that in future papers we will be able to relate this quantity to the $E_{2}$ page without the basepoint filtration.

These theorems seem to give a first step towards categorifying Jaeger's composition product. In \cite{Wagner}, Wagner gave a categorification of the composition product for the $sl_{n}$ polynomials, defined only for fully singular diagrams. If we were to restrict the homology to a fully singular diagram, it closely resembles a natural extension of Wagner's construction to include HOMFLY-PT homlology.

\subsection{Acknowledgements} The author would like to thank Zolt\'{a}n Szab\'{o} for many valuable discussions during the course of this work. The author would also like to thank Ciprian Manolescu, Hao Wu, and Andrew Manion for their helpful suggestions.

\section{Background and Notation}

\subsection{The Composition Product} \label{2.1}

\subsubsection{Jaeger's Definition}

Let $D$ be a diagram for a knot $K$. Viewing $D$ as an oriented 4-valent graph, let $v(D)$ denote the vertices (or crossings) of $D$ and $e(D)$ the edges of $D$. 

Let $f$ be a function from $e(D)$ to $\{1,2\}$ that satisfies the conditions given in the introduction to make it a labeling. Recall that a labeling is admissible if the cycle $f^{-1}(1)$ made no left turns at positive crossings or right turns at negative crossings. For Jaeger's composition product, there is no marked edge, so we won't have to worry about the condition $f(e_{0})=2$.

Given a diagram $D$, consider the diagram obtained by changing each crossing in $D$ to the oriented smoothing. The resulting diagram must be a collection of oriented circles - these are known as the Seifert circles of $D$. We define the \emph{rotation number} $r(D)$ to be sum of the signs of the Seifert circles, with a circle contributing a $+1$ if it is oriented counterclockwise and $-1$ if it is oriented clockwise. Note that when $D$ is a braid, $r(D)$ is simply the negative of the number of strands in the braid, i.e. $r(D)=-b$.

With the HOMFLY-PT polynomial $P_{H}$ as defined in the introduction, let $P'_{H}(a,q,D)=(\frac{a-a^{-1}}{q-q^{-1}})(a^{w(D)})P_{H}(a,q,D)$. Note that $P'_{H}$ is invariant under Reidemeister II and III moves, but performing a Reidemeister I move changes the writhe, so one picks up a factor of $a$ or $a^{-1}$ depending on the sign of the crossing. With this normalization, $P'_{H}(unknot)=\frac{a-a^{-1}}{q-q^{-1}}$, and $P'_{H}(\emptyset)=1$, where $\emptyset$ denotes the empty diagram. Jaeger's composition product can be stated as follows:

\begin{equation}
\label{Jaeger}
 \mathop{\sum_{f \text{ admissible}}} (q-q^{-1})^{\mathcal{T}(f)}a_{1}^{r(D_{f,2})}a_{2}^{-r(D_{f,1})} P'_{H}(a_{1},q,D_{f,1})P'_{H}(a_{2},q,D_{f,2}) = P'_{H}(a_{1}a_{2},q,D)
\end{equation}

The proof of this formula is combinatorial in nature - one can show that it behaves properly under Reidemeister moves and that it satisfies the necessary skein relation via local computations. To complete the proof, we just have to check that it works on the unknot, which is calculated below. For details, see \cite{Jaeger}.

\[ a_{2}^{-1} \frac{a_{1}-a_{1}^{-1}}{q-q^{-1}} + a_{1} \frac{a_{2}-a_{2}^{-1}}{q-q^{-1}} = \frac{a_{1}a_{2}-a_{1}^{-1}a_{2}^{-1}}{q-q^{-1}}=P_{H}'(a_{1}a_{2}, q, unknot) \]

\subsubsection{The Destabilized Composition Product}
\label{dest}

We develop an adaptation of Jaeger's composition product for a decorated diagram $D$, i.e. a diagram that has one marked edge $e_{0}$, in the special case where $a_{1}=q$. The diagram $D$ now has a special Seifert circle, the one containing the marked edge $e_{0}$ - we will call this circle $S_{0}$. We will define the sign of a Seifert circle $S$ as follows:

\begin{equation}
\label{signs}
 sign(S) = \begin{cases} 
      +1 & \textrm{ if $S$ is oriented CCW and $S$ does not contain $S_{0}$} \\
      -1 & \textrm{ if $S$ is oriented CCW and $S$ contains $S_{0}$} \\
      -1 & \textrm{ if $S$ is oriented CW and $S$ does not contain $S_{0}$} \\
      +1 & \textrm{ if $S$ is oriented CW and $S$ contains $S_{0}$} \\
      0 & \textrm{ if $S=S_{0}$} \\
      
   \end{cases} 
\end{equation}

An alternative way to view these signs is to imagine our diagram is in $S^{2}$ instead of the plane, so that each Seifert circle bounds two discs. To determine the sign of the Seifert circle, we view it as the boundary of the disc that does not contain the edge $e_{0}$. Then, as before, we say that it is $+1$ if it is oriented CCW and negative if it is oriented $CW$. The special circle containing $e_{0}$ has no contribution.

With these sign conventions, let $r(D)$ denote the sum of the signs of the Seifert circles. We can define our reduced version of the composition product by 

\begin{equation}
\label{compred}
\begin{split}
  \mathop{\sum_{f \text{ admissible}}}_{f(e_{0})=2} \Big[ (q-q^{-1})^{\mathcal{T}(f)}q^{r(D_{f,2})-s(D_{f,2})} & a^{-r(D_{f,1})-s(D_{f,1})} \\
& \cdot  P_{H}(q,q,D_{f,1})P_{H}(a,q,D_{f,2}) \Big] = P_{H}(qa,q,D)
\end{split}  
\end{equation}

\noindent
where the signs of the Seifert circles in both $D_{f,1}$ and $D_{f,2}$ are given by (\ref{signs}) relative to the marked edge $e_{0}$, even though $e_{0}$ always belongs to $D_{f,2}$. The quantities $s(D_{f,i})$, defined to be the sum of the signs of the crossings in $D$ with at least one adjacent edge labeled $i$, stems from the fact that Jaeger's composition product came with factors of $a^{w(D)}$. Note that $s(D_{f,1})=w(D)-w(D_{f,2})$ and  $s(D_{f,2})=w(D)-w(D_{f,1})$.

The proof of this equality is identical to the proof for Jaeger's, since everything is the same locally. The only differences (aside from the notational difference of removing the shifts by $w(D)$) are that our labelings require that $f(e_{0})=2$ and we don't have the factor of $\frac{a-a^{-1}}{q-q^{-1}}$. Jaeger's calculations show that our construction satisfies the correct skein relation, and that it is invariant under Reidemeister moves that take place away from the marked edge $e_{0}$. By the equivalence of knots and (1,1) tangles, these are the only Reidemeister moves we need to show invariance. Thus, to complete the proof, we just need to check that the formula holds on the base case of the unknot.

There is only one labeling that contributes for the unknot. Since there is only one edge, it must be the marked edge $e_{0}$, and $f(e_{0})=2$. For this labeling $f$, $r(D_{f,1})=r(D_{f,2})=s(D_{f,1})=s(D_{f,2})=0$, $P_{1}(\phi)=\frac{q-q^{-1}}{q-q^{-1}}=1$, and $P_{H}(unknot)=1$, so (\ref{compred}) becomes $1 \cdot 1=1$. This establishes the base case, which proves the formula.

\subsubsection{An Example: The Right Handed Trefoil}
\begin{figure}[h!]

 \centering
   \begin{overpic}[width=.25\textwidth]{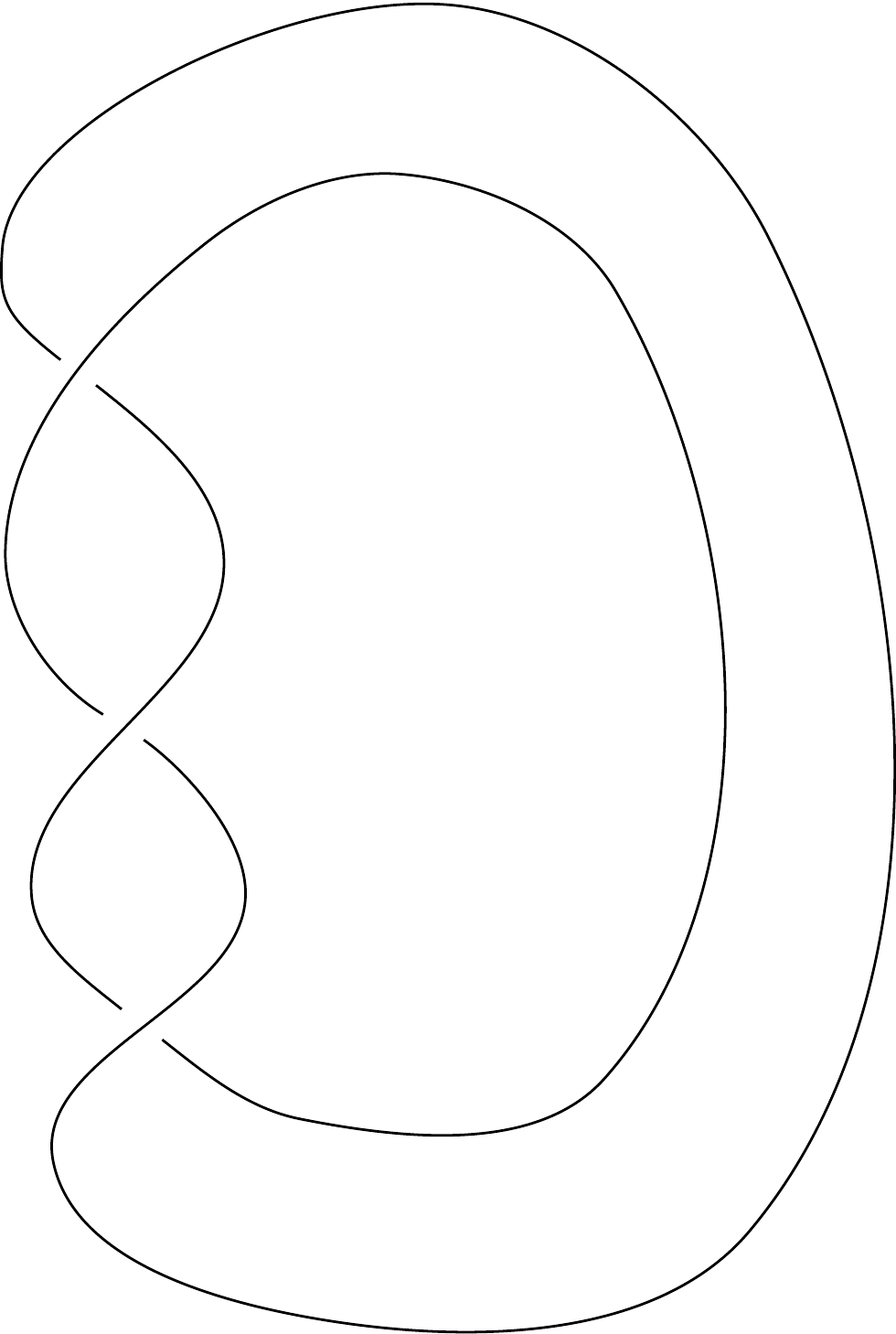}
   \put(-8,83){$e_{0}$}
   \put(10,83){$e_{1}$}
   \put(-7,58){$e_{2}$}
   \put(19,58){$e_{3}$}
   \put(-6,32){$e_{4}$}
   \put(20,32){$e_{5}$}
   \put(-1.5,79){$\circ$}
   \end{overpic}
   
\caption{Braid Diagram for the Right Handed Trefoil}

\end{figure}

This diagram $D$ has four local cycles which we will describe in terms of the edges in $f^{-1}(1)$, since this set uniquely characterizes $f$. These four sets are $\emptyset$, $e_{1}e_{2}e_{5}$, $e_{1}e_{3}e_{4}$, and $e_{1}e_{3}e_{5}$. Their contributions are listed in the table below. The sum of these contributions is $a^{-2}+a^{-2}q^{-4}-a^{-4}q^{-4}$, which is equal to $P_{H}(aq,q,D)$. 

\begin{center}
  \begin{tabular}{ l | c }

    Cycle & Contribution  \\ \hline
    $\emptyset$ & $q^{-4}(a^{-2}q^{2}+a^{-2}q^{-2}-a^{-4})$ \\ \hline
    $e_{1}e_{2}e_{5}$ & $(q-q^{-1})q^{-3}a^{-2}$  \\ \hline
    $e_{1}e_{3}e_{4}$ & $(q-q^{-1})q^{-3}a^{-2}$  \\ \hline
    $e_{1}e_{3}e_{5}$ & $(q-q^{-1})^{3}q^{-3}a^{-2}$  \\ \hline
    \bf{Total} & $a^{-2}+a^{-2}q^{-4}-a^{-4}q^{-4}$ \\ \hline

  \end{tabular}
\end{center}

\subsection{HOMFLY-PT Homology and the $E(-1)$ Spectral Sequence}

\subsubsection{HOMFLY-PT Homology}

In this section we will give a description of HOMFLY-PT homology similar to that of Rasmussen in \cite{Rasmussen}, with the same grading conventions. Let $L$ be a link in $S^{3}$, and $D$ a connected braid diagram for $L$, oriented clockwise. We view $D$ as an oriented 4-valent graph with each vertex decorated with $+$ (positive crossing) or $-$ (negative crossing). 

If $e_{1},...,e_{n}$ are the edges of $D$, let $X_{1},...,X_{n}$ be corresponding indeterminates. At each crossing $c$, we have outgoing edges $e_{i(c)},e_{j(c)}$ and incoming edges $e_{k(c)},e_{l(c)}$. We define the ground ring $R$ by 

\begin{equation}
R=\Z_{2}[X_{1},...,X_{n}]/I
\end{equation}  

\noindent
where $I$ is the ideal generated by $\{X_{i(c)} +X_{j(c)} + X_{k(c)} +X_{l(c)} \}$ over all crossings $c$. The ring $R$ comes equipped with an internal $q$-grading given by $q(X_{i})=2$. We are using $\Z_{2}$ coefficients instead of $\Q$ coefficients so that we don't have to count signs in the knot Floer context. The fact that the resulting homology is still an invariant follows from \cite{Krasner}. Although he only proves that HOMFLY-PT homology gives an invariant with integer coefficients, the argument extends to $\Z_{2}$ via the universal coefficient theorem.

We are now ready to define our complex. There will be three gradings: $gr_{q}$ (the $q$-grading),  $gr_{h}$ (twice the horizontal grading), and $gr_{v}$ (twice the vertical grading). Let $R\{i,j,k\}$ denote the ring $R$ shifted by $i,j,k$ in $gr_{q},gr_{h},gr_{v}$, respectively. Define the complex for positive an negative crossings as in Figures \ref{pos11} and \ref{neg11}.

\vspace{5 mm}

\begin{figure}[!h]
\centering
\begin{tikzpicture}
  \matrix (m) [matrix of math nodes,row sep=5em,column sep=6em,minimum width=2em] {
     R\{0,-2,0\} & R\{0,0,0\} \\
     R\{2,-2,-2\} & R\{0,0,-2\} \\};
  \path[-stealth]
    (m-2-1) edge node [left] {$C_{H}(D_{+})= \hspace{20mm} X_{j}+X_{k}$} (m-1-1)
    (m-1-1) edge node [above] {$X_{k}+X_{i}$} (m-1-2)
    (m-2-2) edge node [right] {$1 \hspace{20mm}$} (m-1-2)
    (m-2-1) edge node [above] {$X_{i}X_{j}+X_{k}X_{l}$} (m-2-2);
\end{tikzpicture}
\caption{The HOMFLY-PT complex for a positive crossing}\label{pos11}
\end{figure}

\vspace{5 mm}

\begin{figure}[!h]
\centering
\begin{tikzpicture}
  \matrix (m) [matrix of math nodes,row sep=5em,column sep=6em,minimum width=2em] {
     R\{0,-2,2\} & R\{-2,0,2\} \\
     R\{0,-2,0\} & R\{0,0,0\} \\};
  \path[-stealth]
    (m-2-1) edge node [left] {$C_{H}(D_{-})= \hspace{31mm}  1$} (m-1-1)
    (m-1-1) edge node [above] {$X_{i}X_{j}+X_{k}X_{l}$} (m-1-2)
    (m-2-2) edge node [right] {$X_{j}+X_{k} \hspace{9mm}$} (m-1-2)
    (m-2-1) edge node [above] {$X_{k}+X_{i}$} (m-2-2);
\end{tikzpicture}
\caption{The HOMFLY-PT complex for a negative crossing}\label{neg11}
\end{figure}

To get the total complex $C_{H}(D)$, we just take the tensor product (over $R$) over all the crossings in $D$.

\begin{equation}
C_{H}(D)= \bigotimes_{c} C_{H}(D_{c})       
\end{equation}

Note that each tensorand $C_{H}(D_{c})$ admits a horizontal and a vertical filtration. Let $d_{+}$ denote the differential consisting of all horizontal arrows and $d_{v}$ the differential of all vertical arrows. We see that $d_{+}$ is homogeneous of degree $\{2,2,0\}$ and $d_{v}$ is homogeneous of degree $\{0,0,2\}$. The total differential $d_{h}+d_{v}$ is not homogeneous with respect to the three gradings, and therefore does not define a triply graded homology theory. Instead, we do the following:

\begin{defn}

The middle HOMFLY-PT homology $H(L)$ of a link $L$ is given by 

\begin{equation}
 H_{H}(L) = H_{\ast}(H_{\ast}(C_{H}(D), d_{+}),d_{v}^{\ast}) \{ -w+b-1, w+b-1, w-b+1\}  
 \end{equation}

\noindent
where $w$ is the writhe of $D$ and $b$ is the number of strands in the braid.

\end{defn}

From this complex we can define the reduced HOMFLY-PT homology $\overline{H}(L)$ by setting one of the $X_{i}$ equal to $0$, or equivalently tensoring $C_{H}(D)$ with the reducing complex

\[ R \{2, 0, -2\}  \xrightarrow{\hspace{3mm}X_{i}  \hspace{3mm} } R \{0, 0, 0\} \]

\noindent
It was shown by Khovanov and Rozansky in \cite{KR2} that $H(L)$ and $\overline{H}(L)$ are link invariants.

\subsubsection{The $E_{k}(-1)$ Spectral Sequence}
\label{h-1}

In \cite{Rasmussen}, Rasmussen identifies a collection of spectral sequences $E_{k}(N)$ for $N \ge 1$ from HOMFLY-PT homology to $sl(N)$ homology. He conjectures that there is a symmetry on HOMFLY-PT homology that would give $E_{k}(-N)$ spectral sequences as well, and constructs an $E_{k}(-1)$ spectral sequence that seems to be an example of this symmetry. In this section we will give a description of this spectral sequence.

We mentioned above that the total differential $d_{+}+d_{v}$ was not homogeneous with respect to all three gradings. However, it still defines a bigraded homology theory.

Let $gr_{M}= \frac{1}{2} (gr_{h}-gr_{v}-2q)$ and $gr_{A}=\frac{1}{2} (gr_{h}-q)$. Then $d_{+}+d_{v}$ is homogeneous of degree $-1$ with respect to $gr_{M}$ and homogeneous of degree $0$ with respect to $gr_{A}$ - let $R\{i,j\}$ denote the ring $R$ shifted by $i,j$ in $gr_{M},gr_{A}$. We define the \emph{unfiltered} HOMFLY-PT complex of a diagram $D$ 

\begin{equation}
C_{-1}(D)=(C_{H}(D), d_{+}+d_{v}) \{w, w\}      
\end{equation}

\noindent
and the $sl_{-1}$ homology by $H_{-1}(L)=H_{\ast}(C_{-1}(D))$. Then HOMFLY-PT homology is constructed as the $E_{2}$ page of the spectral sequence on $C_{-1}(D)$ induced by the vertical filtration, which converges to $H_{-1}(L)$. The same is true in the reduced case. 

\begin{lem}[\hspace{1sp}\cite{Rasmussen}] The $sl_{-1}$ homology of a k-component link is isomorphic to the HOMFLY-PT homology of the k-component unlink.

\end{lem}

\begin{proof}

We will start by rewriting the complexes $C_{H}(D_{+})$ and $C_{H}(D_{-})$ in terms of our new gradings:

\begin{figure}[!h]
\centering
\begin{tikzpicture}
  \matrix (m) [matrix of math nodes,row sep=5em,column sep=6em,minimum width=2em] {
     R\{-1,-1\} & R\{0,0\} \\
     R\{-2,-2\} & R\{1,0\} \\};
  \path[-stealth]
    (m-2-1) edge node [left] {$C_{H}(D_{+})= \hspace{20mm} X_{j}+X_{k}$} (m-1-1)
    (m-1-1) edge node [above] {$X_{k}+X_{i}$} (m-1-2)
    (m-2-2) edge node [right] {$1\hspace{20mm}$} (m-1-2)
    (m-2-1) edge node [above] {$X_{i}X_{j}+X_{k}X_{l}$} (m-2-2);
\end{tikzpicture}
\end{figure}

\begin{figure}[!h]
\centering
\begin{tikzpicture}
  \matrix (m) [matrix of math nodes,row sep=5em,column sep=6em,minimum width=2em] {
     R\{-2,-1\} & R\{1,1\} \\
     R\{-1,-1\} & R\{0,0\} \\};
  \path[-stealth]
    (m-2-1) edge node [left] {$C_{H}(D_{-})= \hspace{31mm} 1$} (m-1-1)
    (m-1-1) edge node [above] {$X_{i}X_{j}+X_{k}X_{l}$} (m-1-2)
    (m-2-2) edge node [right] {$X_{j}+X_{k} \hspace{9mm}$} (m-1-2)
    (m-2-1) edge node [above] {$X_{k}+X_{i}$} (m-2-2);
\end{tikzpicture}
\end{figure}

We will proceed by cancelling the 1 arrows in both complexes, giving resulting complexes

\[ R\{-2,-2\} \xrightarrow{X_{j}+X_{k}} R\{-1,-1\}  \hspace{2mm}\text{    and    }\hspace{2mm} R\{0,0\} \xrightarrow{X_{j}+X_{k}} R\{1,1\}    \]

\noindent
for the positive and negative crossings, respectively. We can remove the overall grading shift of $\{w,w\}$ by modifying them to 

\[ R\{-1,-1\} \xrightarrow{X_{j}+X_{k}} R\{0,0\}  \hspace{2mm}\text{    and    }\hspace{2mm} R\{-1,-1\} \xrightarrow{X_{j}+X_{k}} R\{0,0\}    \]

Since $e_{j}$ and $e_{k}$ are positioned diagonally at the crossing, they lie on the same component. Together with the relations $X_{i}+X_{j}+X_{k}+X_{l}$, they serve to identify all of the edges of each component, with one redundancy for each component beyond the first. Thus, if our link $L$ has $n$ components and we choose an ordering of the edges such that $X_{1},...,X_{n}$ all lie on different components, then we get

\begin{equation}
H_{-1}(L)= \Z_{2} [X_{1},...,X_{n}] \otimes V_{-}^{n-1} 
\end{equation}

\noindent
where $V_{-}= \Z\{0,0\} \bigoplus \Z\{-1,-1\}$. This is precisely the HOMFLY-PT homology of the n-component unlink. To reduce, we simply set $X_{1}=0$ so that

\begin{equation}
\overline{H}_{-1}(L)= \Z_{2} [X_{2},...,X_{n}] \otimes V_{-}^{n-1} 
\end{equation}

\end{proof}

\begin{rem}
While HOMFLY-PT homology has only been proved to be invariant under braidlike Reidemeister moves, it is a valid construction for any diagram $D$. It follows that the $E_{k}(-1)$ spectral sequence is  well defined for non-braid diagrams and will converge to the homology described above.
\end{rem}

\subsection{The Oriented Cube of Resolutions for Knot Floer Homology}
\subsubsection{Defining the Cube of Resolutions}

In this section we will give a brief review of the oriented cube of resolutions for knot Floer homology, introduced with twisted coefficients by Ozsv\'{a}th and Szab\'{o} in \cite{Szabo}. In this paper, they mention a similarity between the specialization to $t=1$ in this setting and HOMFLY-PT homology, which is explored in more detail by Manolescu in \cite{Manolescu}. In both cases, the knot is in decorated braid position to make maps between cycles more well-behaved - we will make the same assumption.

So let $K$ be a knot in $S^{3}$, and $D$ a braid projection for $K$ with one marked edge $e_{0}$, i.e. a decorated braid projection. To each crossing in $D$, we assign the Heegaard diagram shown in Figure \ref{HDCrossing}. Note that if we place $X$'s at $A^{0}$ and $A^{-}$, we get the Heegaard diagram for a negative crossing, and if we place them at $A^{0}$ and $A^{+}$, we get the diagram for a positive crossing. We also stabilize along the edges as necessary (or add insertions, in the language of \cite{Manolescu}), and since the diagram is in $S^{2}$, we need to leave out an $\alpha$ curve and a $\beta$ curve to make it balanced. We do this at the marked edge, as shown in Figure \ref{HDMarkedEdge}. 

\begin{figure}[h!]

 \centering
   \begin{overpic}[width=.75\textwidth]{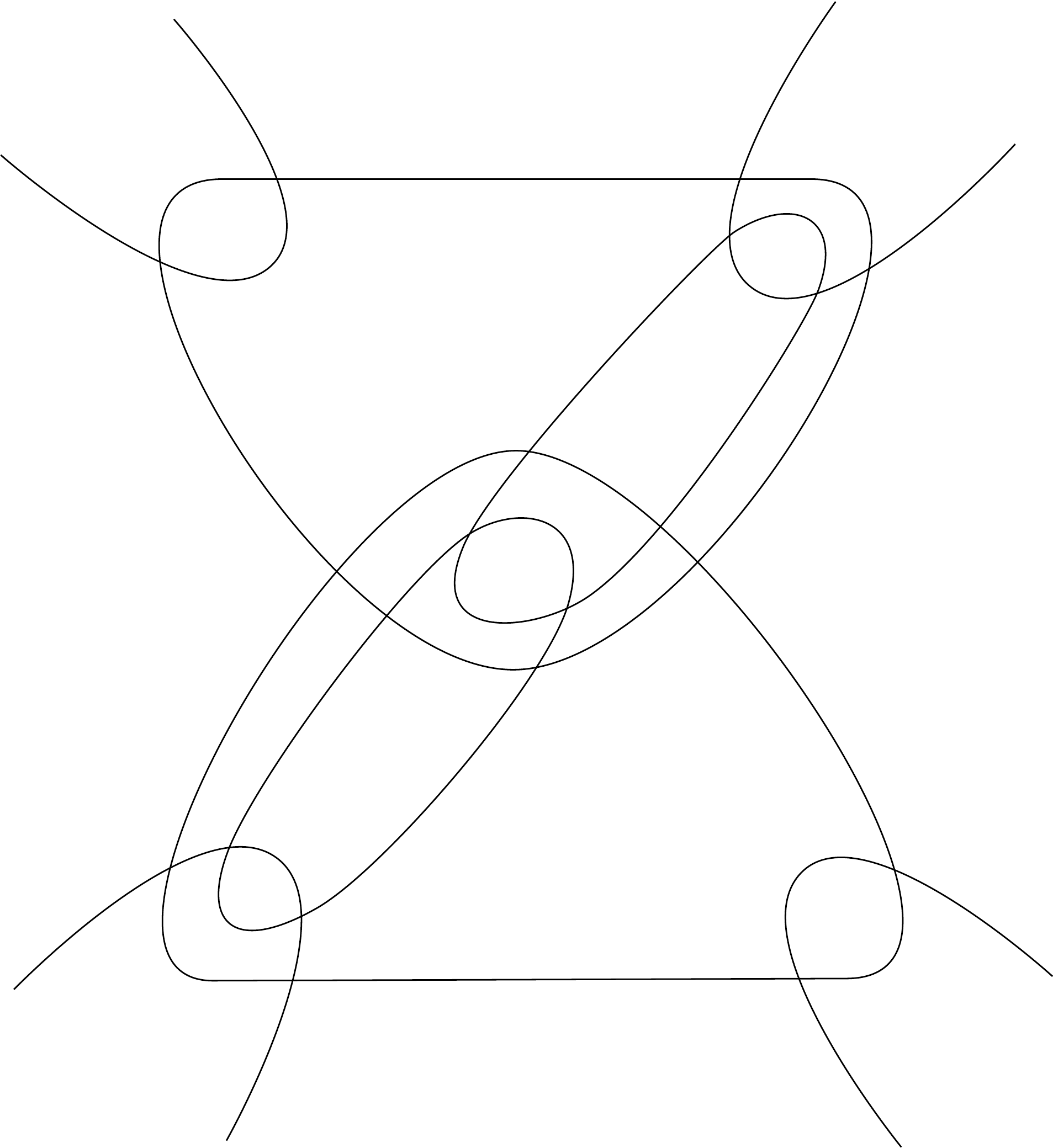}
   \put(39,54){$x$}
   \put(40.4,53){$\bullet$}
   \put(49.3,45){$x'$}
   \put(48.5,46.3){$\bullet$}
   \put(17.6,79){$O_{1}$}
   \put(66,77){$O_{2}$}
   \put(48,55.7){$B$}
   \put(43,49.7){$A_{0}$}
   \put(38.2,44){$B$}
   \put(34,52){$A_{-}$}
   \put(53,48){$A_{+}$}
   \put(21,21.6){$O_{3}$}
   \put(71,19){$O_{4}$}
   \put(43,86){$\alpha_{1}$}
   \put(50,70){$\alpha_{2}$}
   \put(45,12){$\beta_{1}$}
   \put(40,30.4){$\beta_{2}$}
   \end{overpic}
\caption{The Diagram at a Crossing}\label{HDCrossing}
\end{figure}

\begin{figure}[h!]

 \centering
   \begin{overpic}[width=.6\textwidth]{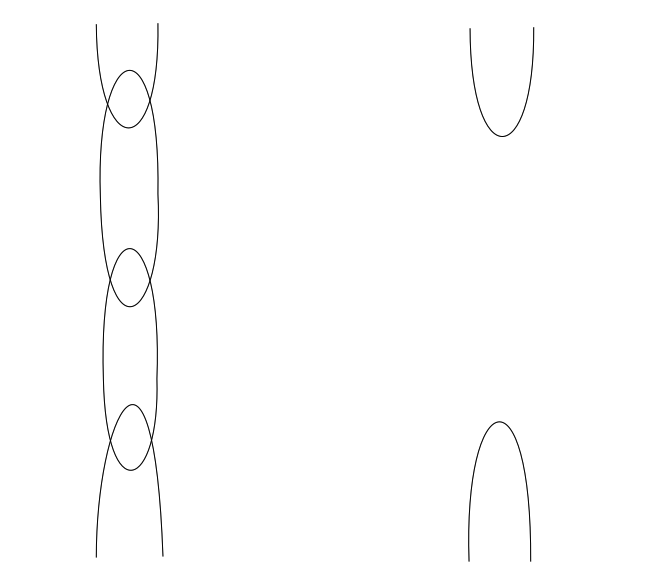}
   \put(17.8,72.5){$O_{1}$}
   \put(75,71){$O_{1}$}
   \put(18.1,44.7){$X$}
   \put(75,44.7){$X$}
   \put(17.9,20){$O_{2}$}
   \put(74.6,18.5){$O_{2}$}
   \end{overpic}
\caption{Heegaard diagrams for an unmarked edge (left) and a marked edge (right)}   \label{HDMarkedEdge}
\end{figure}

Suppose we have a negative crossing, so the $X$ basepoints are at $A^{0}$ and $A^{-}$. It is clear that the set of generators which have $x$ as a coordinate make a subcomplex - call it $X$, and let the quotient complex be $Y$. Then the complex shown in Figure \ref{negcomp} computes the knot Floer homology of $K$, where  $\Phi_{A^{-}}$ counts discs with multiplicity 1 at $A^{-}$ or $A^{0}$ and 0 at $B$, $\Phi_{B}$ counts discs with multiplicity 1 at one of the $B$'s and 0 at $A^{-}$ and $A^{0}$, and $A^{-}B$ counts discs with multiplicity 1 at $A^{-}$ or $A^{0}$ and multiplicity 1 at one of the $B$'s.

\begin{figure}[!h]
\centering
\begin{tikzpicture}
  \matrix (m) [matrix of math nodes,row sep=5em,column sep=6em,minimum width=2em] {
     X & X \\
     Y & X \\};
  \path[-stealth]
    (m-1-1) edge node [left] {\hspace{ 15mm} $\Phi_{A^{-}}$} (m-2-1)
            edge node [above] {1} (m-1-2)
            edge node [right]{$\Phi_{A^{-}B}$} (m-2-2)
    (m-2-1.east|-m-2-2) edge node [below] {$\Phi_{B}$} (m-2-2)
    (m-1-2) edge node [right] {$U_{1}+U_{2}+U_{3}+U_{4}$} (m-2-2);
\end{tikzpicture}
\caption{Complex for the Negative Crossing}
\label{negcomp}
\end{figure}

The quasi-isomorphism can be seen by looking at the vertical filtration, and canceling the top isomorphism. We are left with 

\[ Y \xrightarrow{\Phi_{B}} X \]

\noindent
which is precisely the knot Floer complex.

In order for this construction to make sense, we need two things to hold: first, we need the differential to satisfy $d^{2}=0$ - this was shown by Ozsv\'{a}th and Szab\'{o} in \cite{Szabo} by studying the ends of the Maslov index two holomorphic curves, which give the $U_{1}+U_{2}+U_{3}+ U_{4}$. Second, we need the diagrams to be admissible, which was shown by Manolescu in \cite{Manolescu}.

We can define a similar complex for the positive crossing, taking $X'$ to be the quotient complex of those generators which contain the intersection point $x'$, and $Y'$ the corresponding subcomplex. Then we get a quasi-isomorphism between the knot Floer complex and the one in Figure \ref{poscomp}, once again by imposing the vertical filtration and canceling the isomorphism.

\begin{figure}[!h]
\centering
\begin{tikzpicture}
  \matrix (m) [matrix of math nodes,row sep=5em,column sep=6em,minimum width=2em] {
     X' & Y' \\
     X' & X' \\};
  \path[-stealth]
    (m-1-1) edge node [left] {$U_{1}+U_{2}+U_{3}+U_{4}$} (m-2-1)
            edge node [above] {$\Phi_{B}$} (m-1-2)
            edge node [right]{$\Phi_{A^{+}B}$} (m-2-2)
    (m-2-1.east|-m-2-2) edge node [below] {1} (m-2-2)
    (m-1-2) edge node [right] {$\Phi_{A^{+}}$ \hspace{ 15mm}} (m-2-2);
\end{tikzpicture}
\caption{Complex for the Positive Crossing} \label{poscomp}
\end{figure}

To get the cube of resolutions, we apply the horizontal filtration to these complexes. This filtration corresponds to grading induced by the height in the cube. The complexes 

\[X \xrightarrow{U_{1}+U_{2}+U_{3}+U_{4}}X \hspace{2mm}\text{ and }\hspace{2mm} X' \xrightarrow{U_{1}+U_{2}+U_{3}+U_{4}} X'\]

\noindent
correspond to the singularization of the knot at this crossing, while 

\[ X \xrightarrow{\Phi_{A^{-}}} Y \hspace{2mm} \text{    and    } \hspace{2mm} Y' \xrightarrow{\Phi_{A^{+}}} X' \]

\noindent
correspond to the oriented smoothings. We denote the cube of resolutions complex by $(C_{F}(D), d)$. The differential now decomposes as 

\[ d = d_{0}+d_{1}+...+d_{k} \]

\noindent
where $d_{i}$ increases the cube grading by $i$.

\subsubsection{Generators and Cycles}

Before discussing our filtration, it is worth gaining an understanding of the generators in this complex. The condition that each $\alpha$ and $\beta$ curve have exactly one intersection point allows us to assign to each generator an oriented multi-cycle in the underlying oriented 4-valent graph of the projection $D$. We say that a generator contains an edge $e_{i}$ if it contains an intersection point on one of the small bigons containing $O_{i}$. The sets of intersection points and the corresponding local cycles are described below.

\begin{figure}[h!]

 \centering
   \includegraphics[width=.5\textwidth]{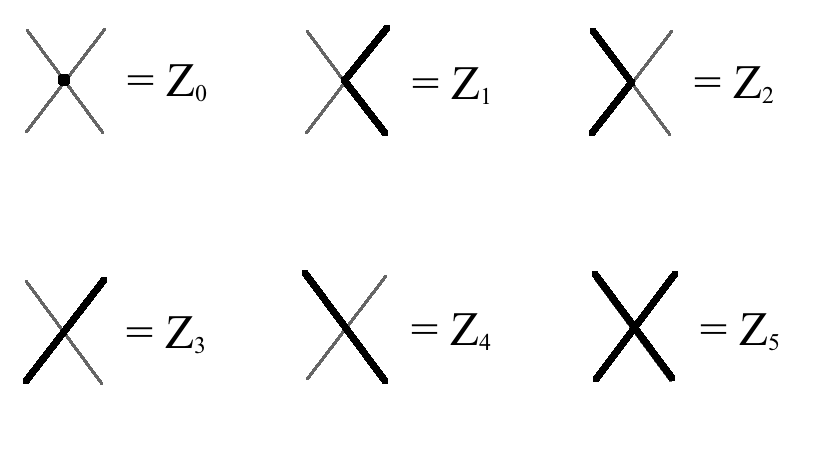}
   \caption{The local cycles at a crossing}
\end{figure}

\begin{figure}[h!]

 \centering
   \begin{overpic}[width=.75\textwidth]{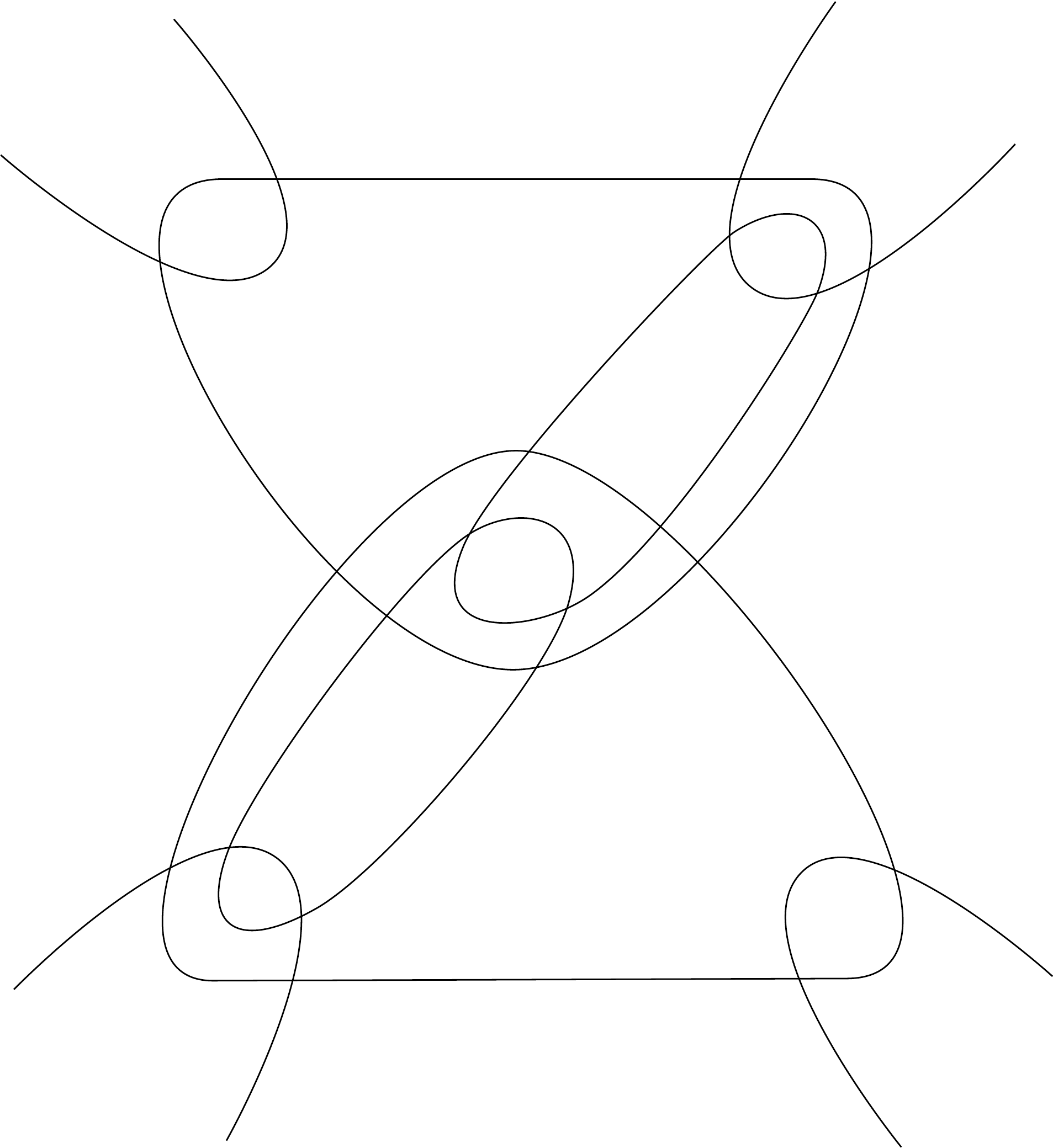}
   \put(38.6,54.6){$e_{1}$}
   \put(40.4,52.9){$\bullet$}
   \put(49.3,45){$e_{2}$}
   \put(48.5,46.3){$\bullet$}
   \put(17.6,79){$O_{1}$}
   \put(66,77){$O_{2}$}
   \put(48,55.7){$B$}
   \put(43,49.7){$A_{0}$}
   \put(38.2,44){$B$}
   \put(34,52){$A_{-}$}
   \put(53,48){$A_{+}$}
   \put(21,21.6){$O_{3}$}
   \put(71,19){$O_{4}$}
   \put(43,86){$\alpha_{1}$}
   \put(50,70){$\alpha_{2}$}
   \put(45,12){$\beta_{1}$}
   \put(40,30.4){$\beta_{2}$}
   \put(13.2,76.4){$\bullet$}
   \put(23.25,83.7){$\bullet$}
   \put(24,85.5){$c_{2}$}
   \put(11.5,75.2){$c_{1}$}
   \put(63.8,83.7){$\bullet$}
   \put(75,76){$\bullet$}
   \put(61.5,85.5){$a_{1}$}
   \put(76.5,75.5){$a_{2}$}
   \put(60.4,79.9){$b_{1}$}
   \put(71.2,71.8){$b_{2}$}
   \put(62.7,78.7){$\bullet$}
   \put(70.2,73.8){$\bullet$}
   \put(44,62.1){$d_{1}$}
   \put(58.5,53.6){$d_{2}$}
   \put(45.4,59.95){$\bullet$}
   \put(56.8,53.45){$\bullet$}
   \put(28.55,49.55){$\bullet$}
   \put(25.5,49.55){$f_{1}$}
   \put(60,50.3){$\bullet$}
   \put(62,50.3){$f_{2}$}
   \put(33,45.5){$\bullet$}
   \put(30,45.65){$g_{1}$}
   \put(46,41.2){$\bullet$}
   \put(47.2,40){$g_{2}$}
   \put(19.25,25.4){$\bullet$}
   \put(17,27.2){$h_{1}$}
   \put(25.4,19.55){$\bullet$}
   \put(26.5,18){$h_{2}$}
   \put(14.1,23.5){$\bullet$}
   \put(12.35,25.1){$i_{1}$}
   \put(24.8,14){$\bullet$}
   \put(25.9,12.34){$i_{2}$}
   \put(68.8,14.1){$\bullet$}
   \put(67,12.6){$j_{1}$}
   \put(77.15,23.3){$\bullet$}
   \put(78.5,25){$j_{2}$}
   \end{overpic}
   \caption{The Labeled Diagram}
   \label{labeled diagram}
\end{figure}

\begin{figure}
\begin{center}
  \begin{tabular}{ l | c }

    Cycles & Generators  \\ \hline
    $Z_{0}$ & $(d,g) \text{ and } (e,f)$ \\ \hline
    $Z_{1}$ & $(a,e,j)\text{ and } (b,g,j)$  \\ \hline
    $Z_{2}$ & $(c,d,h) \text{ and } (c,e,i)$ \\ \hline
    $Z_{3}$ & $(a,e,i) \text{, } (a,d,h)\text{, } (b,f,h) \text{, and } (b,g,i)$  \\ \hline
    $Z_{4}$ & $(c,e,j)$  \\ \hline
    $Z_{5}$ & $(b,c,h,j)$  \\ 
    \hline
  \end{tabular}
\end{center}
\caption{Generators corresponding to each local cycle}
\end{figure}

Every multi-cycle $Z$ except those which include the marked edge $e_{0}$ will have at least one corresponding generator. However, each vertex in the cube of resolutions will each only contain a subset of these generators - for example, if a vertex has a particular crossing smoothed, generators corresponding to $Z_{3}$ and $Z_{4}$ will not appear, while if the crossing is singularized, it is $Z_{5}$ that will be disallowed. The local cycles $Z_{0}$, $Z_{1}$, and $Z_{2}$ will appear in both the singularization and the smoothing, so there will be a non-trivial edge map involving each of these cycles. These are the interesting maps, and they will be discussed in Section \ref{section4}.

In order to draw connections between this complex and the composition product, we will utilize the bijection between multi-cycles and labelings. If $Z$ is a multi-cycle in $D$, then define $f_{Z}$ to be the labeling on $D$ given by 

\begin{equation}
 f_{Z}(e) = \begin{cases} 
      1 & \textrm{ if $e$ is in $Z$} \\
      2 & \textrm{ if $e$ is not in $Z$} \\
      
   \end{cases} 
\end{equation}

At the marked edge $e_{0}$, we are missing one $\alpha$ and one $\beta$ circle. It follows that there are no generators with corresponding cycles containing $e_{0}$, so we will alway have $f(e_{0})=2$, just like in our destabilized composition product formula.

\section{The Basepoint Filtration} \label{basepointsection}
We make our complex into a filtered complex by adding additional basepoints in all of the regions of our Heegaard Diagram that correspond to components of $\R^{2}-D$, labeled with points $p_{i}$ in the Figure \ref{markings}.

\begin{figure}[h!]
 \centering
   \begin{overpic}[width=.75\textwidth]{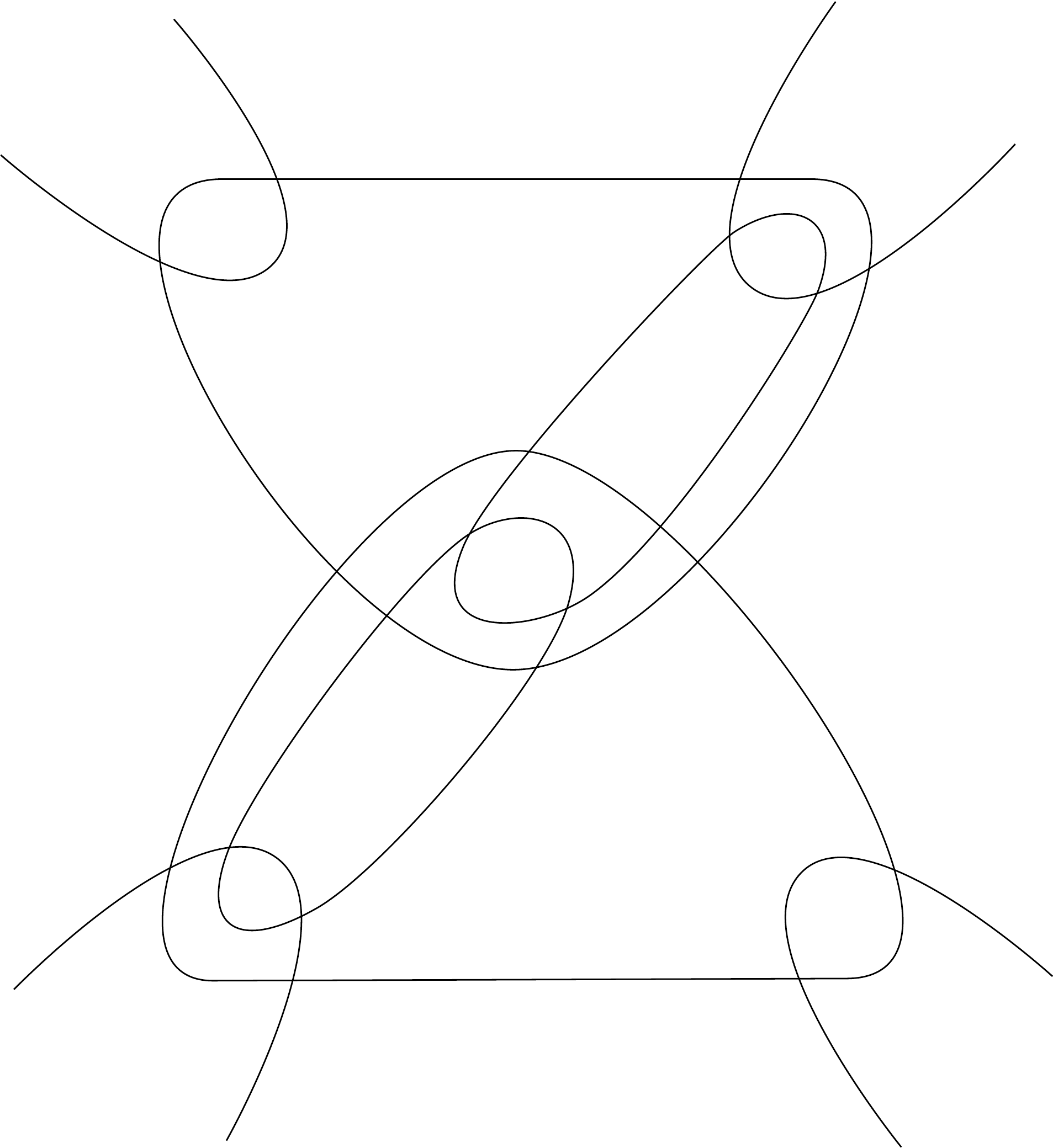}
   \put(39,54){$x$}
   \put(40.4,53){$\bullet$}
   \put(49.3,45){$x'$}
   \put(48.5,46.3){$\bullet$}
   \put(17.6,79){$O_{1}$}
   \put(66,77){$O_{2}$}
   \put(48,55.7){$B$}
   \put(43,49.7){$A_{0}$}
   \put(38.2,44){$B$}
   \put(34,52){$A_{-}$}
   \put(53,48){$A_{+}$}
   \put(21,21.6){$O_{3}$}
   \put(71,19){$O_{4}$}
   \put(43,86){$\alpha_{1}$}
   \put(50,70){$\alpha_{2}$}
   \put(45,12){$\beta_{1}$}
   \put(40,30.4){$\beta_{2}$}
   \put(13,48){$\bullet$}
   \put(80,50){$\bullet$}
   \put(47,5){$\bullet$}
   \put(44,95){$\bullet$}
   \put(10,49){$p_{2}$}
   \put(77,51){$p_{3}$}
   \put(44,6){$p_{4}$}   
   \put(41,96){$p_{1}$}
   \end{overpic}
   \caption{Additional basepoints} \label{markings}
\end{figure}

\begin{lem}

These markings define a filtration on the complex $\mathit{CFK}^{-}(K)$, where the change in filtration level of a differential is given by the sum of the multiplicities of the corresponding holomorphic disc at these basepoints. This filtration does not depend on the location of the X's in the interior regions. 

\end{lem}

\begin{proof}

It is sufficient to show that any periodic domain has multiplicity zero at these markings. This follows from that fact that for any $\alpha$ or $\beta$ circle, the markings and the special X corresponding to the decorated edge lie on the same side. So for any periodic domain, the multiplicity at any of these points is the same as that of the X, which is required to be zero.

\end{proof}

This filtration extends to a filtration on the cube of resolutions complex $C_{F}(D)$, with the $1$ and $U_{1}+U_{2}+U_{3}+U_{4}$ maps preserving the filtration, and we can count multiplicities at the basepoints for $\Phi_{A}$, $\Phi_{B}$, and $\Phi_{AB}$. Define $d^{f}$ to be the differential in the cube of resolutions that preserves the basepoint filtration.

Let $C_{F}(Z_{i})$ denote the complex generated by the elements corresponding to the cycle $Z_{i}$.

\begin{lem}

The differential $d^{f}$ preserves $C_{F}(Z_{i})$, i.e. it does not change the underlying cycle of a generator.

\end{lem}

\begin{proof}

Each basepoint gives a filtration on our complex corresponding to a region in the knot projection. Let $x$ be a generator with multi-cycle $Z$, and let $C$ be an oriented 2-chain with boundary $Z$. If we require that $C$ has multiplicity $0$ on the regions adjacent to the marked edge, it is clear that this 2-chain is unique.

Within the planar Heegaard diagram for $K$, we can find a disc that connects $x$ to a generator corresponding to the empty cycle (see \cite{Szabo}, section 3), and since each region in the knot projection contains a basepoint, the multiplicities of the disc at each basepoint will be equal to the multiplicity of $C$ in that region. Thus, each filtration level uniquely determines a 2-chain $C$, whose boundary gives the multi-cycle $Z$. Since no two multi-cycles correspond to the same 2-chain, this completes the proof.

\end{proof}

It follows that each multi-cycle is in its own filtration level, so the homology of the associated graded object is given by

\begin{equation}
 \bigoplus_{Z} H_{*}(C_{F}(Z), d^{f})  
\end{equation}

\noindent
In the next section we will compute the homology of an arbitrary multi-cycle $Z_{i}$ in terms of the types its local cycles it has at each positive and negative crossing.

\section{The Complex of a Labeling} \label{section4}

\subsection{Complete Resolutions}

Before discussing the whole complex corresponding to a labeling $f$, let's consider what the complex looks like at a vertex in the cube of resolutions, i.e. the complex corresponding to a complete resolution $S$. If $X$ denotes the set of singularized crossings in $S$, then this complex is given by 

\begin{equation}
\label{sing}
C_{F}(S) = \mathit{CFK}^{-}(S) \otimes (\bigotimes_{c \in X} R \xrightarrow{U_{i(c)}+U_{j(c)}+U_{k(c)}+U_{l(c)}} R )
\end{equation}

\noindent
where $\mathit{CFK}^{-}(S)$ denotes the complex coming from the planar Heegaard Diagram for $S$. This complex admits the same basepoint filtration, and the generators correspond to cycles in the same way as the complex for a knot. (They have to, since each generator in the complex for $S$ is also going to be a generator in the complex for $K$.) There are strictly fewer possible cycles in $S$ than in $K$, since if $c$ is singularized then there is no generator of type $Z_{5}$ at $c$, and if $c$ is smoothed then there are no generators of type $Z_{3}$ or $Z_{4}$ at $c$. As always, $Z$ can not include the marked edge $e_{0}$.

The complex corresponding to a cycle $Z$ is easy to compute. For each edge $e_{i}$ in $Z$, there are two choices for the intersection point in the Heegaard diagram, and they can be connected by a bigon containing $O_{i}$. This gives a Koszul complex on the $U_{i}$ for $e_{i}$ in $Z$, which form a regular sequence. We can therefore cancel all of these bigons, setting the corresponding $U_{i}$ equal to zero.

Consider the diagram obtained by deleting the edges in $Z$ from $S$ - call this $S-Z$. We claim that the homology corresponding to the cycle $Z$ is the HOMFLY-PT homology of the singular diagram $S-Z$. This can be seen by examining what happens locally for each possible local cycle at a singularized crossing. If $Z$ is the empty cycle $Z_{0}$, there are two possible generators, and they are connected by a pair of bigons that give us the quadratic map

\[R \xrightarrow{U_{i(c)}U_{j(c)}+U_{k(c)}U_{l(c)}} R\]

\noindent
If we quotient by the linear relations coming from the Koszul complex in (\ref{sing}), this is exactly the HOMFLY-PT complex of a singularization.

If $Z$ has two of the 4 edges at $c$, i.e. $Z=Z_{1}$, $Z_{2}$, $Z_{3}$, or $Z_{4}$, then after canceling the bigons corresponding to those two edges, we have just one generator. Setting those two edges equal to zero, the linear term $U_{i(c)}+U_{j(c)}+U_{k(c)}+U_{l(c)}$ is now just a sum of the remaining two edges. This is precisely the HOMFLY-PT complex assigned to bivalent vertex separating these two edges. Thus, we have the following:

\begin{thm}

If $K$ is a knot in decorated braid position, and $S$ is a complete resolution of $K$, then the homology corresponding to a cycle $Z$ in $K$ at the vertex of the cube of resolutions corresponding to $S$ is given by $H_{H}(S-Z)$. Thus, the basepoint filtered homology of (\ref{sing}) is given by 

\begin{equation}
\label{sing2}
H_{*}(C_{F}(S), d^{f}) \cong  \bigoplus_{Z} H_{H}(S-Z)
\end{equation}

\end{thm}

In terms of labelings, cycles $Z$ in $S$ are in bijection with labelings $f$ for which $H_{1}(S_{f,1})$ is non-trivial, with the bijection given by $Z \mapsto f_{Z}$ as it was previously for non-singular diagrams. Thus, (\ref{sing2}) can be rewritten as 

\begin{equation}
H_{*}(C_{F}(S), d^{f}) \cong \bigoplus_{f} H_{1}(S_{f,1}) \otimes H_{H}(S_{f,2})
\end{equation}

\subsubsection{An Example}

Let $S$ be the complete resolution given in Figure \ref{singularDiagram}. $S$ has three cycles - $\emptyset$, $e_{2}e_{3}$, and $e_{2}e_{4}$. There are two singular points, so the total complex will be given by $CFK^{-}(S)$ tensored with a Koszul complex on the two generators $U_{1}+U_{2}+U_{3}+U_{4}$ and $U_{3}+U_{4}+U_{5}+U_{2}$.

\begin{figure}[h!]

 \centering
   \begin{overpic}[width=.9\textwidth]{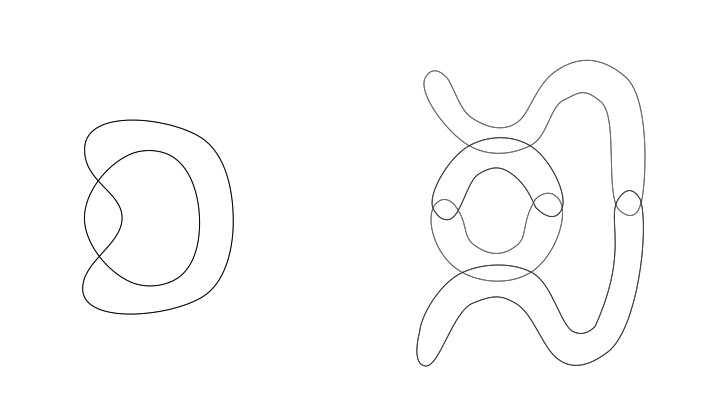}
   \put(9,35){$e_{1}$}
   \put(17,35){$e_{2}$}
   \put(9,25){$e_{3}$}
   \put(17,25){$e_{4}$}
   \put(59,43){$\scalebox{.65}O_{1}$}
   \put(85.5,27){$\scalebox{.65}O_{2}$}
   \put(60,26){$\scalebox{.65}O_{3}$}
   \put(74,26.75){$\scalebox{.65}O_{4}$}  
   \put(57.75,7){$\scalebox{.65}O_{5}$} 
   \put(67.3,34.75){$\scalebox{.65}{XX}$} 
   \put(67,17){$\scalebox{.65}{XX}$} 
   \put(11,33){$\circ$}
   \end{overpic}
\caption{A Complete Resolution and the Corresponding Heegaard Diagram}   \label{singularDiagram}
\end{figure}

The empty cycle $\emptyset$ has four generators in the diagram, and the total complex is given in Figure \ref{emptycycle}. It's clear that tensoring with the Koszul complex, we get the HOMFLY-PT complex corresponding to $S$. For the cycle $e_{2}e_{3}$, we get the complex in Figure \ref{e2e3}, so canceling those arrows sets $U_{2}=U_{3}=0$. The Koszul complex then sets $U_{1}=U_{4}=U_{5}$, so the homology is isomorphic to the HOMFLY-PT homology of the unknot. The cycle $e_{2}e_{4}$ is similar. 

\begin{figure}[h!]

\centering
\begin{tikzpicture}
  \matrix (m) [matrix of math nodes,row sep=5em,column sep=6em,minimum width=2em] {
     R & R \\
     R & R \\};
  \path[-stealth]
    (m-1-1) edge node [left] {$U_{1}U_{2}+U_{3}U_{4}$} (m-2-1)
            edge node [above] {$U_{3}U_{4}+U_{5}U_{2}$} (m-1-2)
    (m-2-1.east|-m-2-2) edge node [below] {$U_{3}U_{4}+U_{5}U_{2}$} (m-2-2)
    (m-1-2) edge node [right] {$U_{1}U_{2}+U_{3}U_{4}$} (m-2-2);
\end{tikzpicture}
\caption{The Empty Cycle}\label{emptycycle}
\vspace{10mm}
\end{figure}

\begin{figure}[h!]

\centering
\begin{tikzpicture}
  \matrix (m) [matrix of math nodes,row sep=5em,column sep=6em,minimum width=2em] {
     R & R \\
     R & R \\};
  \path[-stealth]
    (m-1-1) edge node [left] {$U_{2}$} (m-2-1)
            edge node [above] {$U_{3}$} (m-1-2)
    (m-2-1.east|-m-2-2) edge node [below] {$U_{3}$} (m-2-2)
    (m-1-2) edge node [right] {$U_{2}$} (m-2-2);
\end{tikzpicture}
\caption{The Cycle $e_{2}e_{3}$}\label{e2e3}
\end{figure}

The filtered homology at the vertex in the cube of resolutions corresponding to $S$ is given by 

\[H_{H}(S) \oplus H_{H}(U) \oplus  H_{H}(U) \] 

\noindent
where $U$ is the unknot.

\subsection{The Whole Complex}

Let $c_{R}^{+}$ and $c_{R}^{-}$ denote the set of positive crossings and negative crossings, respectively, at which $Z$ has the local cycle $Z_{1}$ - we will call these `right turns.' Similarly, define $c_{L}^{+}$ and $c_{L}^{-}$ to be the crossings at which $Z$ has the local cycle $Z_{2}$, or `left turns.'

\begin{thm}

\label{cyclecomp}

The complex $C_{F}(Z)$ corresponding to a multi-cycle $Z$ is acyclic if $f_{Z}$ is not admissible or if $Z$ contains the marked edge $e_{0}$. Otherwise,

\[
C_{F}(Z)=(C_{H}(D_{f_{Z},2}), d_{+}+d_{v}) \otimes (\bigotimes_{c \in c_{R}^{+}} R\xrightarrow{U_{i(c)}} R) \otimes   (\bigotimes_{c \in c_{L}^{-}} R \xrightarrow{U_{l(c)}} R) 
\]

\end{thm}

\begin{proof}

We have computed the complex at each vertex in the cube of resolutions - now we just need to count the discs passing through the $B$ basepoints to compute the edge maps. It is not hard to see that if we don't allow discs to pass though the $p_{i}$, then no discs will pass through multiple $B$ basepoints. Thus, there are only edge maps and no higher differentials.

We will now go through each of the possible local cycles at the crossing $c$. The full computation for a negative crossing is included in the appendix, but we will give a summary here.

For the empty cycle $Z_{0}$ (both positive and negative crossing), we get precisely the HOMFLY-PT complex corresponding to a crossing, with the maps preserving the cube filtration giving $d_{+}$ and the edge maps giving $d_{v}$. There are no higher face maps that preserve the basepoint filtration.

For the right turn $Z_{1}$, we see that deleting $Z_{1}$ from the singularization of $c$ leaves the same diagram as deleting $Z_{1}$ from the oriented smoothing of $c$. The corresponding complexes are thus isomorphic. If $c$ is a positive crossing, then the map from the singularization to the smoothing is given by multiplication by $U_{i(c)}$. However, if it is a negative crossing, the edge map from the smoothing to the singularization is the canonical isomorphism.

The left turn $Z_{2}$ is the opposite of $Z_{1}$. We still get isomorphic complexes at the singularization and smoothing of $c$, but the maps between them are switched: for a positive crossing, the edge map is isomorphism, while for a negative crossing, it is multiplication by $U_{l(c)}$.

As mentioned above, the cycles $Z_{3}$, $Z_{4}$ and $Z_{5}$ only appear one of the two resolutions - $Z_{3}$ and $Z_{4}$ in the singularization, and $Z_{5}$ in the smoothing. For $Z_{3}$ and $Z_{4}$, we get a trivial or acyclic complex at the smoothed vertex, and at the singularized vertex we get the HOMFLY-PT complex of the singularization with the cycle removed. For $Z_{5}$, we get a trivial complex at the singularized vertex and the HOMFLY-PT complex of the complete resolution with the edges of the cycle $Z_{5}$ removed at the smoothing.

\end{proof}

\begin{rem}

What we actually get after making the above cancellations is a lift of HOMFLY-PT homology from the ring $Z[U_{1},...,U_{n}]/I$ to $Z[U_{1},...,U_{n}]$ tensored with a Koszul complex on $\{U_{i(c)}+U_{j(c)} + U_{k(c)} + U_{l(c)} \}$. This is the version of HOMFLY-PT homology first introduced in \cite{KR2}. The marked edge makes it so that the linear relations form a regular sequence, which is why we get the middle HOMFLY-PT homology instead of the unreduced version.

\end{rem}

\begin{lem} \label{reducedLemma}

If $f^{-1}(1) \ne \emptyset$, then for each component of $D_{f,2}$ there is an edge $e_{i}$ in that component such that 

\[R \xrightarrow{U_{i}} R \]

\noindent
is one of the terms in one of the two Koszul complexes in (\ref{cyclecomp}).

\end{lem}

\begin{proof}

If a component $C$ of $D_{f,2}$ made no turns, (i.e. never had local cycle $Z_{1}$ or $Z_{2}$ at a crossing) then by starting at the marked edge and following $C$ in a manner consistent with its orientation, we will trace the whole knot $K$. But if $C=K$ then $f^{-1}(1) = \emptyset$.

\end{proof}

This lemma tells us that in the reduced complex, each component of the link $D_{f,2}$ gets reduced at least once, and in the minus complex the same is true except when $f^{-1}(1)=\emptyset$. In the HOMFLY-PT complex, multiplication by two edges on the same component of a link are homotopic. It follows that if we reduce the HOMFLY-PT homology once on each component, it doesn't matter which edge we pick on each component, the resulting homology will be the same. This invariant is known as the totally reduced HOMFLY-PT homology.

This chain homotopy between edges on the same component also tells us that multiplication by an edge on a component that has already been reduced is always trivial, so reducing $k$ times on the same component will give $2^{k-1}$ copies of the homology obtained by reducing once. If $D$ is the diagram of an $m$ component link, and $A$ is a set of $k$ edges of $D$ with at least one edge on each component, then the HOMFLY-PT homology of $D$ reduced at the edges in $A$ consists of $2^{k-m}$ copies of the totally reduced HOMFLY-PT homology of $D$. Since this result depends only on $k$ and not the particular edges in $S$, we will write this as 

\[H_{H}(D,k)\]

We can divide the differentials in Theorem \ref{cyclecomp} into those which preserve the cube filtration (vertex maps) and those which change it by one (edge maps). The differentials which preserve the cube filtration are given by $d_{+}$, and the edge maps are the $d_{v}$ differentials as well as the differentials in the two Koszul complexes (the reducing maps). Letting $d_{i}^{f}$ denote those differentials which preserve the basepoint filtration and increase the cube grading by $i$, the vertex maps are given by $d_{0}^{f}$ and the edge maps are given by $d_{1}^{f}$.

\begin{thm}
\label{homthm}

If $D$ is a decorated braid diagram for a knot $K$ and $C_{F}(D)$ is the complex coming from the oriented cube of resolutions, then 

\begin{equation}
\label{19}
H_{*}(H_{*}(C_{F}(D),d^{f}_{0}),(d^{f}_{1})^{*})=  \mathop{\bigoplus_{f \text{ admissible}}}_{f(e_{0})=2}  H_{1}(D_{f,1}) \otimes H_{H} (D_{f,2}, \mathcal{T}(f)) 
\end{equation}

\noindent
and there is a spectral sequence from this complex to $\mathit{CFK}^{-}(K)$. If $\overline{C}_{F}(D)$ is the reduced complex, then 

\begin{equation}
\label{20}
H_{*}(H_{*}(\overline{C}_{F}(D),d^{f}_{0}),(d^{f}_{1})^{*})=  \mathop{\bigoplus_{f \text{ admissible}}}_{f(e_{0})=2}  H_{1}(D_{f,1}) \otimes \overline{H}_{H} (D_{f,2},\mathcal{T}(f)) 
\end{equation}

\noindent
and there is a spectral sequence from this complex to $\widehat{\mathit{HFK}}(K)$.

\end{thm}

\begin{proof}
This theorem follows from Theorem \ref{cyclecomp} together with the above discussion. 
\end{proof}

The reason why we include the $H_{1}(D_{f,1})$ even though it is always one dimensional is not just because it draws analogies with the composition product - it also provides information regarding where these homologies lie in the cube of resolutions. Given a labeling $f$, the HOMFLY-PT homology $H_{H}(D_{f,2})$ is only going to appear when all of the crossings in $D_{f,1}$ are smoothed. This corresponds with the fact that $sl_{1}$ homology is trivial whenever there is a singular point in $D_{f,1}$, and the contribution comes from all crossings having the oriented smoothing.

The homology in (\ref{20}) is finite dimensional, and with the proper grading shift for each labeling $f$ the graded Euler characteristic is given by $P_{H}(aq,q,D)$. (This follows directly from the composition product formula.) We will discuss how the triple grading on this complex relates to the Maslov and Alexander gradings in the next section.

Going back to Theorem \ref{cyclecomp}, we can see what would happen if we forgot about the cube filtration. The homology corresponding to a labeling $f$ would then be $H_{1}(D_{f,1}) \otimes H_{-1}(D_{f,2},T(f))$, which we see by replacing the HOMFLY-PT homology with the $sl_{-1}$ homology. This gives a true categorification of the (1,-1) composition product, since the spectral sequence induced by the basepoint filtration converges to knot Floer homology, which categorifies the Alexander polynomial $P_{0}(q,D)$.

\begin{thm}
\label{knotfloer}

Let $C_{F}(D)$ be the complex given by the cube of resolutions for knot Floer homology. Then,

\begin{equation}
H_{*}(C_{F}(D),d^{f}_{0}+d^{f}_{1})= \mathop{\bigoplus_{f \text{ admissible}}}_{f(e_{0})=2} H_{1}(D_{f,1}) \otimes H_{-1} (D_{f,2},\mathcal{T}(f))
\end{equation}

\noindent
or, in the case of reduced knot Floer homology,

\begin{equation}
H_{*}(\overline{C}_{F}(D),d^{f}_{0}+d^{f}_{1})=  \mathop{\bigoplus_{f \text{ admissible}}}_{f(e_{0})=2} H_{1}(D_{f,1}) \otimes H_{-1} (D_{f,2},\mathcal{T}(f)+1)= \mathop{\bigoplus_{f \text{ admissible}}}_{f(e_{0})=2} V^{\otimes \mathcal{T}(f)} 
 \end{equation}

\noindent
where $V$ is a two-dimensional vector space over $Z_{2}$. There are differentials on these complexes giving $\mathit{HFK}^{-}(K)$ and $\widehat{\mathit{HFK}}(K)$, respectively.

\end{thm}

\subsection{Gradings} \label{gradings} In this section we will discuss the gradings on complexes in the previous two theorems.

\subsubsection{A Bigrading for Theorem \ref{knotfloer}}

We will first discuss the version where the cube filtration is ignored, as this will allow us to build up to the triply graded version. Knot Floer homology comes equipped with two gradings - the Maslov grading and the Alexander Grading. The Maslov grading is the homological grading, i.e. the differential decreases the Maslov grading by 1, while the Alexander grading is preserved. The variables $U_{i}$ all have Maslov grading $-2$ and Alexander grading $-1$. Let $\{i,j\}$ denote a shift by $i$ in Maslov grading and $j$ in Alexander grading. Note that these gradings align with the bigrading in 

Let $x$ and $y$ be generators in the knot Floer complex. If there is a differential $\delta x = P(U_{1},...,U_{n}) y$, then the above information tells us that $P$ is a homogeneous polynomial, and if the degree of $P$ is $n$, then 

\begin{equation}
M(x)-M(y)= 1 -2n \hspace{2mm} \text{  and  } \hspace{2mm}  A(x)-A(y)=-n
\end{equation}

\noindent
 where $M$ and $A$ are the Maslov and Alexander gradings, respectively.

The grading for each labeling will be easiest to describe in terms of the cycle $Z=f^{-1}(1)$. Let us begin with the empty cycle $Z_{\phi}$. The absolute Maslov and Alexander gradings are difficult to pin down, so we will define a grading that is correct up to an overall shift. It was shown in the proof of Theorem 3.3 that after canceling the linear relations $U_{i}+U_{j}+U_{k}+U_{l}$, we get the complex $(C_{H}(D), d_{+}+d_{v})$ from Section \ref{h-1}, ignoring gradings. However, as noted above, the relative gradings in the $E_{k}(-1)$ complex are the same as those in the knot Floer complex. Since we are only define a relatively graded theory, we can apply those gradings to the complex corresponding to the empty cycle.

Thus, to each positive crossing, we get the complex in Figure \ref{pos} and to each negative crossing the complex in Figure \ref{neg}.

\begin{figure}[!h]

\centering
\begin{tikzpicture}
  \matrix (m) [matrix of math nodes,row sep=5em,column sep=6em,minimum width=2em] {
     R\{0,0\} & R\{1,1\} \\
     R\{-1,-1\} & R\{2,1\} \\};
  \path[-stealth]
    (m-2-1) edge node [left] {$U_{j}+U_{k}$} (m-1-1)
    (m-1-1) edge node [above] {$U_{k}+U_{i}$} (m-1-2)
    (m-2-2) edge node [right] {$1$} (m-1-2)
    (m-2-1) edge node [above] {$U_{i}U_{j}+U_{k}U_{l}$} (m-2-2);
\end{tikzpicture}
 \caption{The Positive Crossing}
 \label{pos}
 \vspace{10mm}
\end{figure}

\begin{figure}[!h]
\centering
\begin{tikzpicture}
  \matrix (m) [matrix of math nodes,row sep=5em,column sep=6em,minimum width=2em] {
     R\{-3,-2\} & R\{0,0\} \\
     R\{-2,-2\} & R\{-1,-1\} \\};
  \path[-stealth]
    (m-2-1) edge node [left] {$1$} (m-1-1)
    (m-1-1) edge node [above] {$U_{i}U_{j}+U_{k}U_{l}$} (m-1-2)
    (m-2-2) edge node [right] {$U_{j}+U_{k}$} (m-1-2)
    (m-2-1) edge node [above] {$U_{k}+U_{i}$} (m-2-2);
\end{tikzpicture}
\caption{The Negative Crossing}
\label{neg}
\end{figure}

With these conventions, the generator of the homology $\Z_{2}[U]$ will be in grading $(0,0)$. We will now apply the Maslov index 1 discs used by Ozsv\'{a}th and Szab\'{o} to determine the gradings of the non-empty cycles relative to the empty cycle.

\textbf{Special Case:} Assume first that the cycle $Z$ is homeomorphic to $S^{1}$. Since it lies in $S^{2}$, it bounds two discs - let $\mathbb{D}$ denote the one such that $Z$ is oriented clockwise as the boundary of $\mathbb{D}$. Assume for now that the marked edge is not contained in $\mathbb{D}$. 

Define $\mathcal{A}$ to be the set of vertices in $Z$, with $\mathcal{T}$ the set of vertices at which $Z$ makes a turn (local cycles $Z_{1}$ or $Z_{2}$ from figure 4), and $\mathcal{D}$ the set of vertices at which it is a diagonal (locally $Z_{3}$ or $Z_{4}$). We will further divide the sets $\mathcal{T}$ and $\mathcal{D}$ into positive and negative crossings, which we will write as $\mathcal{T}_{+}, \mathcal{T}_{-}, \mathcal{D}_{+}$, and $\mathcal{D}_{-}$. In an abuse of notation, we will use the same symbols for the orders of these sets. 

For the crossings not contained in $Z$, the same cancelations will be taking place in $Z$ and $Z_{\phi}$ and the homology will end up in the same gradings. However, for the crossings in $\mathcal{A}$, the complex is given by a Koszul complex on the edges in $Z$, which form a regular sequence. Therefore, the homology will lie in the bottom of the Koszul complex (the lowest algebraic grading).

Our cycle will only appear in the cube of resolutions where the vertices in $\mathcal{D}$ are singularized, but will appear where the vertices in $\mathcal{T}$ are smoothed or singularized. For each such choice at the vertices in $T$, we will get a copy of the complex $H_{*}(C_{H}(D-Z), d_{+}+d_{v})$. Let us look at the case where all of the vertices in $\mathcal{A}$ are singularized. Then, at a positive crossing, $C_{F}(Z_{\phi})$ will have complex 

\begin{equation} \label{25}
R\{-1,-1\} \xrightarrow{U_{i}U_{j}+U_{k}U_{l}} R\{2,1\}
\end{equation}

\noindent
and at a negative crossing 

\begin{equation} \label{26}
R\{-3,-2\} \xrightarrow{U_{i}U_{j}+U_{k}U_{l}} R\{0,0\} 
\end{equation}

In \cite{Szabo}, Ozsv\'{a}th and Szab\'{o} identify two discs from the bottom generator of the Koszul complex on the edges in $Z$ to the bottom generator of the corresponding Koszul complex on $C_{F}(Z_{\phi})$. From (\ref{25}) and (\ref{26}), we know that the latter has grading $\{2D_{+} + 2T_{+}, D_{+} + T_{+} \}$

The order of these two terms in the $\{U_{i}\}$ is given by $\mathcal{T}_{+} + \frac{\mathcal{D}}{2} $. (Note that $\mathcal{D}$ is always even, so $ \frac{\mathcal{D}}{2} $ is a non-negative integer.) Applying (equation), we get that the lowest grading in the Koszul complex in $C_{F}(Z)$ is given by 

\begin{equation}
 \{2\mathcal{D}_{+} + 2\mathcal{T}_{+} - \mathcal{D} - 2\mathcal{T}_{+} +1, \mathcal{D}_{+} + \mathcal{T}_{+} - \mathcal{T}_{+} -  \frac{\mathcal{D}}{2} \}
\end{equation}

\begin{equation}
 = \{\mathcal{D}_{+}-\mathcal{D}_{-}, \frac{\mathcal{D}_{+}-\mathcal{D}_{-}}{2} \}
\end{equation}

We have made the choice to singularize all of the crossings in $\mathcal{T}$. However, as noted above, we will also have a copy of $(C_{H}(D-Z), d_{+}+d_{v})$ if we smooth any number of them, giving us $2^{\mathcal{T}}$ copies. They are arranged in a Koszul complex, with each edge map given by multiplication by one of the $U_{i}$, so they differ in grading by $\{1,1\}$. For the negative crossings $\mathcal{T}_{-}$, the maps go from the smoothing to the singularization, while for the positive crossings $\mathcal{T}_{+}$ they go from the singularization to the smoothing. Therefore, the total complex corresponding to the cycle $Z$ is given by 

\begin{equation}
(C_{H}(D-Z), d_{+}+d_{v}) \otimes V_{+}^{\otimes \mathcal{T}_{+}} \otimes V_{-}^{\otimes \mathcal{T}_{-}} \{\mathcal{D}_{+}-\mathcal{D}_{-} +1 , \frac{\mathcal{D}_{+}-\mathcal{D}_{-}}{2} \}
\end{equation}

\begin{equation}
= (C(D-Z), d_{h}+d_{v}) \otimes V^{\otimes \mathcal{T}} \{\mathcal{D}_{+}-\mathcal{D}_{-} +\frac{\mathcal{T}_{+}-\mathcal{T}_{-}}{2} +1, \frac{\mathcal{D}_{+}-\mathcal{D}_{-}+\mathcal{T}_{+}-\mathcal{T}_{-}}{2} \}
\end{equation}

\noindent
where $V_{+}=\Z_{2}\{0,0\} \oplus \Z_{2}\{1,1\}$, $V_{-}=\Z_{2}\{0,0\} \oplus \Z_{2}\{-1,-1\}$, and $V=\Z_{2}\{-\frac{1}{2},-\frac{1}{2}\} \oplus \Z_{2}\{\frac{1}{2},\frac{1}{2}\}$.
From Section 2.2, we know that the homology of $ (C_{H}(D-Z), d_{+}+d_{v})$ is given by 

\begin{equation}
\Z_{2} [U_{1},...,U_{l}] \bigotimes V_{-}^{l-1} 
\end{equation}

\noindent
where $l$ is the number of components of the link $D-Z$, and each $U_{i}$ lies on a different component. Since we are dealing with a knot, each component of $D-Z$ will be reduced by one of the edge maps coming from a turn in $Z$ (Theorem 3.3). Thus, as our final homology for the cycle $Z$ we get

\begin{equation}
V^{\otimes \mathcal{T} - 1}  \{\mathcal{D}_{+}-\mathcal{D}_{-} +\frac{\mathcal{T}_{+}-\mathcal{T}_{-}}{2}+1, \frac{\mathcal{D}_{+}-\mathcal{D}_{-}+\mathcal{T}_{+}-\mathcal{T}_{-}}{2} \}
\end{equation}

\noindent
or, in the reduced case,
\begin{equation}
V^{\otimes T }  \{D_{+}-D_{-} +\frac{T_{+}-T_{-}}{2}+1, \frac{D_{+}-D_{-}+T_{+}-T_{-}}{2} \}
\end{equation}

In this computation, we have made the assumption that the marked edge is not contained in the disc $\mathbb{D}$. If this is not the case, then we still have a Maslov index 1 discs as described above, with the only difference that they now passes through one $X$ basepoint and one $O$ basepoint. This shifts the grading of elements in our cycle by $\{-2,0\}$, so we get the new grading shift 

\begin{equation}
V^{\otimes \mathcal{T} }  \{\mathcal{D}_{+}-\mathcal{D}_{-} +\frac{\mathcal{T}_{+}-\mathcal{T}_{-}}{2}-1, \frac{\mathcal{D}_{+}-\mathcal{D}_{-}+\mathcal{T}_{+}-\mathcal{T}_{-}}{2} \}
\end{equation}

\textbf{General Formula:} Let $Z$ be an arbitrary multicycle in the diagram $D$. In the special case calculated above, we were able to assume all of the crossings contained in $Z$ were singularized. That is no longer the case for a general multicycle - those vertices at which $Z$ has the local cycle $Z_{5}$ must be smoothed for generators corresponding to $Z$ to appear. Let $\mathcal{X}$ denote the set of crossings at which $Z$ has this multicycle, which we will as usual divide into positive crossings $\mathcal{X}_{+}$ and negative crossings $\mathcal{X}_{-}$. 

After smoothing the crossings in $\mathcal{X}$ and singularizing those in $\mathcal{T}$ and $\mathcal{D}$, we can view our multi-cycle as a union of $n$ copies of $S^{1}$. We will apply the same technique used in the one-component case for each of these cycles. Due to the smoothings, the grading of the generator that we will be mapping to is given by 

\begin{equation}
\{ 2\mathcal{D}_{+}+2\mathcal{T}_{+}+\mathcal{X}_{+}-\mathcal{X}_{-}, \mathcal{D}_{+}+\mathcal{T}_{+}+\mathcal{X}_{+}-\mathcal{X}_{-} \}
\end{equation}

The $n$ (clock-wise oriented) discs that map to the empty cycle will have a total polynomial order of $\mathcal{T}_{+} +\frac{\mathcal{D}}{2}$, but now have Maslov index $n$. Let $k_{+}$ denote the number of discs which do not contain the marked edge, and $k_{-}$ the number of discs which do contain the marked edge, so $k_{+}+k_{-}=n$. Then the grading of the generator that the discs are mapping from is given by 

\begin{equation}
\{ \mathcal{D}_{+}-\mathcal{D}_{-}+\mathcal{X}_{+}-\mathcal{X}_{-}+k_{+}-k_{-}, \frac{\mathcal{D}_{+}-\mathcal{D}_{-}}{2}+\mathcal{X}_{+}-\mathcal{X}_{-} \}
\end{equation}

\noindent
Applying the same arguments as in the special case, we get the homology of a multicycle to be 

\begin{equation}
\label{bggrd}
V^{\otimes \mathcal{T}}  \{\mathcal{D}_{+}-\mathcal{D}_{-} +\mathcal{X}_{+}-\mathcal{X}_{-} +k_{+}-k_{-}+\frac{\mathcal{T}_{+}-\mathcal{T}_{-}}{2}, \frac{\mathcal{D}_{+}-\mathcal{D}_{-}+\mathcal{T}_{+}-\mathcal{T}_{-}}{2} +\mathcal{X}_{+}-\mathcal{X}_{-} \}
\end{equation}

Although we have so far assumed we are working with a knot, the arguments extend quite easily to links - the only difference is that when we want a reduced version, we will have to reduce the additional components.

To make these gradings absolute, we need to add an overall shift of $\{-w(D)-r(D),-\frac{1}{2}[w(D)-r(D)]\}$. The absolute Maslov grading is determined by the generator of the homology of $CF^{-}(S^{3})$ obtained by setting the $X$'s in the Heegaard diagram equal to $1$. We see that in the cube of resolutions, this generator corresponds to resolution with all crossings smoothed, and the multi-cycle consisting of all the circles not containing the marked edge. By the above computations, the Maslov grading shift of this cycle at this vertex in the cube is $w(D)-r(D)$, which is why we need an overall shift of $-w(D)+r(D)$. The Alexander grading shift follows from the fact that the Euler characteristic needs to be the Alexander polynomial, which determines it uniquely. We will soon show that this grading shift does indeed give the Alexander polynomial.

If $f$ is the labeling of $D$ such that $f^{-1}(1)=Z$, then with our new grading shifts (\ref{bggrd}) becomes 

\begin{equation}
V^{\otimes \mathcal{T} }  \{-w(D_{f,2})+r(D_{f,2})-\frac{\mathcal{T}_{+}-\mathcal{T}_{-}}{2}, \frac{1}{2}(w(D_{f,1})-w(D_{f,2}) + r(D)\}
\end{equation}

We can see that this formula is beginning to resemble the composition product. Knot Floer homology is related to the Alexander polynomial $P_{0}(q,K)$ via the formula

\begin{equation}
\label{poincare}
P_{0}(q,K) = \sum_{i,j} (-1)^{i}q^{2j} dim(\widehat{\mathit{HFK}}_{i,j}(K)) 
\end{equation}

\noindent
where $\widehat{\mathit{HFK}}_{i,j}(K)$ denotes the knot Floer homology in Maslov grading $i$ and Alexander grading $j$.

If we replace $\widehat{\mathit{HFK}}$ in (\ref{poincare}) with our complex, we get the sum

\begin{equation}
\mathop{\sum_{f \text{ admissible}}}_{f(e_{0})=2} (-1)^{-w(D_{f,2})+r(D_{f,2})+\mathcal{T}_{-}(f)} (q-q^{-1})^{\mathcal{T}(f)}q^{r(D)+w(D_{f,1})-w(D_{f,2})} 
\end{equation}

\noindent
We can show that this is equivalent to the (1,-1) composition product formula via two combinatorial identities. The first is that $w(D_{f,1})-w(D_{f,2})=s(D_{f,1})-s(D_{f,2})$, since the number of crossings that have two of the four edges labeled 1 must be the same as the number of crossings with two of the four edges labeled 2. The second is that given a decorated diagram $D$ of an n-component link $L$, $w(D)+r(D) \equiv n+1 \text{ }(\text{ }mod\text{ } 2\text{ })$. This follows from the base case of the n-component unlink, together with invariance under Reidemeister moves and changing the sign of a crossing.

Plugging in these identities, we get 

\begin{equation}
\mathop{\sum_{f \text{ admissible}}}_{f(e_{0})=2} (-1)^{\mathcal{T}_{-}(f)} (q-q^{-1})^{\mathcal{T}(f)}q^{r(D)+s(D_{f,1})-s(D_{f,2})} (-1)^{\#(D_{f,2})+1} 
\end{equation}

\noindent
where $\#(D)$ is the number of components in the underlying link. But $P_{-1}(D)=(-1)^{\#(D_{f,2})+1}$ and $P_{1}$ is identically $1$, so this equation becomes

\begin{equation}
\mathop{\sum_{f \text{ admissible}}}_{f(e_{0})=2} (-1)^{\mathcal{T}_{-}(f)} (q-q^{-1})^{\mathcal{T}(f)}q^{r(D)+s(D_{f,1})-s(D_{f,2})} P_{1}(q,D_{f,1})P_{-1}(q,D_{f,2})
\end{equation} 

\noindent
which is precisely our formula for the composition product in Section \ref{dest}. It is therefore equal to the Alexander polynomial $P_{0}(q,D)$. Thus, with this choice of Heegaard diagram, knot Floer homology can be viewed as a categorification of the $(1,-1)$ composition product. This also proves that our choice of absolute Alexander grading is correct.

\subsubsection{A Triple Grading for Theorem \ref{homthm}}

The complex $C_{F}(D)$ comes equipped with three gradings: the Maslov grading, the Alexander grading, and the grading induced by the cube of resolutions. The differentials $d^{f}_{0}$ and $d^{f}_{1}$ are clearly homogeneous with respect to these three gradings. To be consistent with the previous sections, we will double the cube grading so that it takes only one value mod 2. (It can either be even or odd, depending on the overall grading shift discussed below.) Call these three gradings $gr_{M}$, $gr_{A}$, and $gr_{v}$ respectively.

We will be relating these gradings to the triple grading $(gr_{q},gr_{h},gr_{v})$ on HOMFLY-PT homology used in \cite{Rasmussen}. This grading has the minor drawback that its graded Euler characteristic differs from the one we have been using via mirroring:

\begin{equation}
 \sum_{i,j,k} (-1)^{(k-j)/2}q^{i}a^{j} dim(\overline{H}_{H}^{i,j,k}(K)) = P_{H}(a,q,m(D))
\end{equation}

\noindent
where  $\overline{H}_{H}^{i,j,k}$ denotes the reduced HOMFLY-PT homology in $gr_{q}=i$,  $gr_{h}=j$, and $gr_{v}=k$, and $m(D)$ is the mirror of the diagram $D$. Since this is the grading convention most widely used when discussing spectral sequences on HOMFLY-PT homology, it is worthwhile to frame our triple grading in this perspective, despite the difference in chirality from our conventions.

We discussed the Maslov and Alexander gradings on HOMFLY-PT complexes in the previous section, and our cube grading $gr_{v}$ is defined the same way as the HOMFLY-PT vertical grading, which is why we gave them the same name. The $q$-grading and the horizontal grading are related to these three gradings in the following way.

\begin{equation}
\label{transform}
 gr_{q}=2gr_{A}-2gr_{M}-gr_{v} 
\end{equation}

\begin{equation}
\label{transform2}
gr_{h}=4gr_{A}-2gr_{M}-gr_{v} 
\end{equation}

\noindent
The grading shifts corresponding to each labeling $f$ were computed for the Alexander and Maslov gradings in the previous section, but we still need to compute the difference between the vertical grading of $H_{H}(D_{f,2})$ and the overall vertical grading.

The vertical grading on HOMFLY-PT homology of a diagram $D$ is centered on the smoothings, with an overall grading shift of $w(D)-b(D)+1$, where $b(D)$ is the number of strands in $D$. We will assume that the marked edge $e_{0}$ is on the leftmost strand, so that $-b(D)+1=r(D)$. Thus, the vertical grading of the vertex in the cube in which all crossings are smoothed is $w(D)+r(D)$.

We will take our vertical grading to be centered on the smoothings as well, with the same overall shift of $w(D)+r(D)$. Let $f$ be a labeling of $D$. The difference between the vertical grading on the HOMFLY-PT complex corresponding to $D_{f,2}$ and the overall vertical grading is  then given by $w(D)+r(D)-w(D_{f,2})-r(D_{f,2}) - 2\mathcal{D}_{+}(f)+2\mathcal{D}_{+}(f)-\mathcal{T}_{+}(f)+\mathcal{T}_{-}(f)$, which can be simplified to 

\[ gr_{v} \text{ shift }= w(D)+ r(D_{f,1})-s(D_{f,2})-\mathcal{D}_{+}(f)+\mathcal{D}_{-}(f) \] 

Since each $\mathcal{T}_{+}$ corresponds to a reducing complex $R\{-2\} \xrightarrow{U_{i}} R\{0\} $ and each $\mathcal{T}_{-}$ corresponds to  $R\{0\} \xrightarrow{U_{i}} R\{2\} $ where the given gradings are the vertical gradings, we have included a shift of $-\mathcal{T}_{+}(f)+\mathcal{T}_{-}(f)$ in the vertical grading to make it so that the reducing complex corresponding to both $\mathcal{T}_{+}$ and $\mathcal{T}_{-}$ is given by 

\[
R\{-1\} \xrightarrow{U_{i}} R\{1\} 
\]

\noindent
From the previous section, we have that 

\[ gr_{M} \text{ shift } = -w(D_{f,2})+r(D_{f,2}) - \frac{\mathcal{T}_{+}-\mathcal{T}_{-}}{2} \]

\noindent
and 

\[ gr_{A} \text{ shift } = \frac{1}{2} (w(D_{f,1})-w(D_{f,2})+r(D)) \]

\noindent
Defining $gr_{q}$ and $gr_{h}$ as in (\ref{transform}) and (\ref{transform2}), we get the corresponding grading shifts to be 

\[ gr_{q} \text{ shift } = -r(D_{f,2})+s(D_{f,2}) \]

\noindent
and 

\[ gr_{h} \text{ shift }=r(D_{f,1}) +s(D_{f,1}) \]

\noindent
We can now reformulate Theorem \ref{homthm} in terms of our triple grading $(gr_{q},gr_{h},gr_{v})$.

\begin{thm}

\label{gradedhomflythm}

Let $H_{H}(D)$ denote the triply graded HOMFLY-PT homology, with the grading conventions given in \cite{Rasmussen}, and define the homology of $H_{1}$ to be in grading $\{0,0,0\}$. Then
\[
H_{*}(H_{*}(\overline{C}_{F}(D),d^{f}_{0}), (d^{f}_{1})^{*})=  \mathop{\bigoplus_{f \text{ admissible}}}_{f(e_{0})=2}  H_{1}(D_{f,1}) \otimes \overline{H}_{H} (D_{f,2},\mathcal{T}(f)) \{q(f), h(f), v(f) \}
\]

\noindent
where $q(f)=-r(D_{f,2})+s(D_{f,2})$, $h(f)=r(D_{f,1}) +s(D_{f,1})$, and $v(f)=w(D)+ r(D_{f,1})-s(D_{f,2})-\mathcal{D}_{+}(f)+\mathcal{D}_{-}(f)$.

\end{thm}

The $\mathcal{T}(f)$ edges at which the HOMFLY-PT homology is reduced correspond to triply graded complexes

\[
R\{1,0,-1\} \xrightarrow{U_{i}} R\{-1,0,1\} 
\]

Let this triply graded complex be denoted $E_{2}(\overline{C}^{f}_{F}(D))$, since it is the $E_{2}$ page of the spectral sequence induced by the cube filtration on the basepoint filtered complex. Applying the composition product formula, we will show the following:

\begin{thm}
\label{poincarehomfly}
Let $E^{i,j,k}_{2}(\overline{C}^{f}_{F}(D))$ denote the homology lying in triple grading $(i,j,k)$ with respect to the triple grading $(gr_{q}, gr_{h}, gr_{v})$. Then

\[
 \sum_{i,j,k} (-1)^{(k-j)/2}q^{i}a^{j} dim(\overline{H}^{i,j,k}(D)) = P_{H}(aq,q,m(D))
\]

\end{thm}

\begin{proof}

We will start by modifying the reducing complexes in a minor way - we want their graded Euler characteristic to give $q-q^{-1}$ - this can be achieved by shifting the vertical grading by $\mathcal{T}(f)$, so that it is given by 

\[
R\{1,0,0\} \xrightarrow{U_{i}} R\{-1,0,2\} 
\]

\noindent
This means we have to subtract $\mathcal{T}(f)$ from the vertical grading shift, making it equal to 

\[w(D) + r(D_{f,1})-s(D_{f,2})-\mathcal{D}_{+}(f)+\mathcal{D}_{-}(f)-\mathcal{T}_{+}(f)-\mathcal{T}_{-}(f)\]

\noindent
The vertical shift minus the horizontal shift can be computed to be $-2\mathcal{D}_{+}(f)+2\mathcal{D}_{-}(f)-2\mathcal{T}_{+}(f)$. The graded Euler characteristic is then given by 

\[ 
\mathop{\sum_{f \text{ admissible}}}_{f(e_{0})=2} (-1)^{-\mathcal{D}_{+}(f)+\mathcal{D}_{-}(f)-\mathcal{T}_{+}(f)} (q-q^{-1})^{\mathcal{T}(f)}q^{-r(D_{f,2})+s(D_{f,2})}a^{r(D_{f,1}) +s(D_{f,1})} P_{H}(a,q,m(D))
\]

\noindent
Since $\mathcal{D}_{+}+\mathcal{D}_{-}$ is always even, we can simplify this expression to 

\[ 
\mathop{\sum_{f \text{ admissible}}}_{f(e_{0})=2} (-1)^{\mathcal{T}_{+}(f)} (q-q^{-1})^{\mathcal{T}(f)}q^{r(D_{f,2})+s(D_{f,2})}a^{-r(D_{f,1}) +s(D_{f,1})} P_{H}(a,q,m(D_{f,2}))
\]

\noindent
If we let $x=a^{-1}$ and $z=q^{-1}$, then since $P_{H}(a^{-1},q^{-1},m(D))=P_{H}(a,q,D)$, this formula becomes 

\[ 
\mathop{\sum_{f \text{ admissible}}}_{f(e_{0})=2} (-1)^{\mathcal{T}_{-}(f)} (z-z^{-1})^{\mathcal{T}(f)}z^{-r(D_{f,2})-s(D_{f,2})}x^{r(D_{f,1}) -s(D_{f,1})} P_{H}(x,z,D)
\]

\noindent
which, by the composition product, is equal to $P_{H}(xz,z,D)$. Substituting back for $a$ and $q$, we get that the graded Euler characteristic is

\[
 \sum_{i,j,k} (-1)^{(k-j)/2}q^{i}a^{j} dim(\overline{H}^{i,j,k}(D)) = P_{H}(a^{-1}q^{-1},q^{-1},D) = P_{H}(aq,q,m(D))
\]

\end{proof}

\appendix

\section{Computations For the Negative Crossing}

In this section, we describe the computations for Theorem \ref{cyclecomp} in the negative crossing case (the positive crossing is similar). The labelling for coordinates is given in Figure \ref{labeled diagram}. With this labeling of coordinates, the subcomplex $X$ is generated by elements containing $e_{1}$. To distinguish the three copies of $X$ in our complex, we will call the upper left copy $X^{\alpha}$, the upper right $X^{\beta}$, and the lower right $X^{\gamma}$. The $e_{1}$ coordinate of the generators will be written $e_{1}^{\alpha}$, $e_{1}^{\beta}$, and $e_{1}^{\gamma}$ respectively.

\begin{figure}[!h]
\centering
\begin{tikzpicture}
  \matrix (m) [matrix of math nodes,row sep=5em,column sep=6em,minimum width=2em] {
     X^{\alpha} & X^{\beta} \\
     Y & X^{\gamma} \\};
  \path[-stealth]
    (m-1-1) edge node [left] {\hspace{ 15mm} $\Phi_{A^{-}}$} (m-2-1)
            edge node [above] {1} (m-1-2)
            edge node [right]{$\Phi_{A^{-}B}$} (m-2-2)
    (m-2-1.east|-m-2-2) edge node [below] {$\Phi_{B}$} (m-2-2)
    (m-1-2) edge node [right] {$U_{1}+U_{2}+U_{3}+U_{4}$} (m-2-2);
\end{tikzpicture}
\caption{Complex for the Negative Crossing}
\label{negativecomplex}
\end{figure}

Since the filtered differentials are easy to count and do not depend on the complex structure, we will simply list them and go forward with computations rather than going into detail regarding the actual counting of the holomorphic discs.

\newpage

\subsection{$Z_{0}$ - The Empty Cycle}

For the empty cycle, $X$ is generated by $e_{1}f_{1}$ and $e_{1}f_{2}$, and $Y$ is generated by $e_{2}f_{1}$, $e_{2}f_{2}$, $d_{1}g_{1}$, $d_{1}g_{2}$, $d_{2}g_{1}$, and $d_{2}g_{2}$. The total complex is given by

\begin{align*}
e_{x}^{\alpha}f_{2} & \mapsto e_{y}f_{2} +(U_{1} U_{2} + U_{3}U_{4})e_{x}^{\alpha}f_{1} + e_{x}^{\beta}f_{2} +U_{1}d_{2}g_{1}+U_{4}d_{1}g_{2}\\
 e_{x}^{\alpha}f_{1} & \mapsto e_{y}f_{1} +e_{x}^{\beta}f_{1} +d_{1}g_{1} \\
 e_{y}f_{2} & \mapsto (U_{1} U_{2} + U_{3}U_{4})e_{y}f_{1} + (U_{2}+U_{3})e_{x}^{\gamma}f_{2}\\
 e_{y}f_{1} & \mapsto (U_{2}+U_{3})e_{x}^{\gamma}f_{1}\\
 d_{2}g_{2} & \mapsto e_{y}f_{2} + U_{2}d_{1}g_{2} + U_{3}d_{2}g_{1}\\
d_{1}g_{2} & \mapsto U_{3}d_{1}g_{1} + e_{x}^{\gamma}f_{2} +U_{1}e_{y}f_{1}\\
d_{2}g_{1} & \mapsto U_{2}d_{1}g_{1} + e_{x}^{\gamma}f_{2} +U_{4}e_{y}f_{1}\\
d_{1}g_{1} & \mapsto (U_{1}+U_{4})e_{x}^{\gamma}f_{1}\\
e_{x}^{\beta}f_{2} & \mapsto (U_{1} U_{2} + U_{3}U_{4})e_{x}^{\beta}f_{1} +(U_{1} +U_{2} + U_{3}+U_{4})e_{x}^{\gamma}f_{2} \\
e_{x}^{\beta}f_{1} & \mapsto (U_{1} +U_{2} + U_{3}+U_{4})e_{x}^{\gamma}f_{1} \\
e_{x}^{\gamma}f_{2}  & \mapsto (U_{1} U_{2} + U_{3}U_{4})e_{x}^{\gamma}f_{1}\\
e_{x}^{\gamma}f_{1} & \mapsto 0\\
\end{align*}

We will cancel the two arrows $e_{x}^{\alpha}f_{2} \mapsto e_{y}f_{2}$ and $e_{x}^{\alpha}f_{1} \mapsto e_{y}f_{1}$. Doing so yields the complex

\begin{align*}
d_{2}g_{2} & \mapsto  e_{x}^{\beta}f_{2} +(U_{1}+U_{3})d_{2}g_{1} + (U_{2}+U_{4})d_{1}g_{2} \\
d_{1}g_{2} & \mapsto (U_{1}+U_{3})d_{1}g_{1} + e_{x}^{\gamma}f_{2} +U_{1}e_{x}^{\beta}f_{1}\\
d_{2}g_{1} & \mapsto (U_{2}+U_{4})d_{1}g_{1} + e_{x}^{\gamma}f_{2} +U_{4}e_{x}^{\beta}f_{1}\\
d_{1}g_{1} & \mapsto (U_{1}+U_{4})e_{x}^{\gamma}f_{1}\\
e_{x}^{\beta}f_{2} & \mapsto (U_{1} U_{2} + U_{3}U_{4})e_{x}^{\beta}f_{1} +(U_{1} +U_{2} + U_{3}+U_{4})e_{x}^{\gamma}f_{2} \hspace{ 9mm } \\
e_{x}^{\beta}f_{1} & \mapsto (U_{1} +U_{2} + U_{3}+U_{4})e_{x}^{\gamma}f_{1} \\
e_{x}^{\gamma}f_{2}  & \mapsto (U_{1} U_{2} + U_{3}U_{4})e_{x}^{\gamma}f_{1}\\
e_{x}^{\gamma}f_{1} & \mapsto 0\\
\end{align*}

With change of basis $d_{1}g_{2} \mapsto d_{1}g_{2}+d_{2}g_{1}$, the ($d$, $g$) generators become the HOMFLY-PT complex of a resolution, the $e_{x}$ generators are the HOMFLY-PT complex of the singularization, and the map between them is precisely the zip homomorphism. 

\newpage

\subsection{$Z_{1}$ - the Right Cycle}

For the cycle $Z_{1}$, we have the generators $(a,e,j)$ and $(b,g,j)$. The total complex is given by

\begin{align*}
a_{2}e_{1}^{\alpha}j_{2} & \mapsto U_{2}a_{1}e_{1}^{\alpha}j_{2} + U_{4}a_{2}e_{1}^{\alpha}j_{1} + a_{2}e_{2}j_{2} +a_{2}e_{1}^{\beta}j_{2} +U_{1}b_{2}g_{1}j_{2} \\
a_{1}e_{1}^{\alpha}j_{2} & \mapsto U_{4}a_{1}e_{1}^{\alpha}j_{1} +a_{1}e_{2}j_{2} + a_{1}e_{1}^{\beta}j_{2} + U_{1}b_{1}g_{1}j_{2}\\
a_{2}e_{1}^{\alpha}j_{1} & \mapsto U_{2}a_{1}e_{1}^{\alpha}j_{1} +a_{2}e_{2}j_{1} + a_{2}e_{1}^{\beta}j_{1} + U_{1}b_{2}g_{1}j_{1} + a_{2}e_{1}^{\gamma}j_{2}\\
a_{1}e_{1}^{\alpha}j_{1} & \mapsto a_{1}e_{2}j_{1} +a_{1}e_{1}^{\beta}j_{1} + U_{1}b_{1}g_{1}j_{1} + a_{1}e_{1}^{\gamma}j_{2}\\
b_{2}g_{2}j_{2} & \mapsto U_{2}b_{1}g_{2}j_{2}+ U_{4}b_{2}g_{2}j_{1} + U_{3}b_{2}g_{1}j_{2} +a_{2}e_{2}j_{2}\\
b_{1}g_{2}j_{2} & \mapsto U_{4}b_{1}g_{2}j_{1} + U_{3}b_{1}g_{1}j_{2} + a_{2}e_{1}^{\gamma}j_{2}+a_{1}e_{2}j_{2}\\
b_{2}g_{2}j_{1} & \mapsto U_{2}b_{1}g_{2}j_{1} + U_{3}b_{2}g_{1}j_{1}+a_{2}e_{2}j_{1}\\
b_{1}g_{2}j_{1} & \mapsto U_{3}b_{1}g_{1}j_{1} +a_{2}e_{1}^{\gamma}j_{1}+a_{1}e_{2}j_{1}\\ 
b_{2}g_{1}j_{2} & \mapsto U_{2}b_{1}g_{1}j_{2}+ U_{4}b_{2}g_{1}j_{1} + a_{2}e_{1}^{\gamma}j_{2}\\ 
b_{1}g_{1}j_{2} & \mapsto U_{4}b_{1}g_{1}j_{1} + a_{1}e_{1}^{\gamma}j_{2}\\
b_{2}g_{1}j_{1} & \mapsto U_{2}b_{1}g_{1}j_{1} + a_{2}e_{1}^{\gamma}j_{1}\\
b_{1}g_{1}j_{1} & \mapsto a_{1}e_{1}^{\gamma}j_{1} \\
a_{2}e_{2}j_{2} & \mapsto U_{2}a_{1}e_{2}j_{2} + U_{4}a_{2}e_{2}j_{1} + (U_{2}+U_{3})a_{2}e_{1}^{\gamma}j_{2}\\
a_{1}e_{2}j_{2} & \mapsto U_{4}a_{1}e_{2}j_{1}+ (U_{2}+U_{3})a_{1}e_{1}^{\gamma}j_{2}\\
a_{2}e_{2}j_{1} & \mapsto U_{2}a_{1}e_{2}j_{1} + (U_{2}+U_{3})a_{2}e_{1}^{\gamma}j_{1} \\
a_{1}e_{2}j_{1} & \mapsto (U_{2}+U_{3})a_{1}e_{1}^{\gamma}j_{1}\\
a_{2}e_{1}^{\beta}j_{2} & \mapsto U_{2}a_{1}e_{1}^{\beta}j_{2} + U_{4}a_{2}e_{1}^{\beta}j_{1} +(U_{1}+U_{2}+U_{3}+U_{4})a_{2}e_{1}^{\gamma}j_{2}\\
a_{1}e_{1}^{\beta}j_{2} & \mapsto U_{4}a_{1}e_{1}^{\beta}j_{1}+(U_{1}+U_{2}+U_{3}+U_{4})a_{1}e_{1}^{\gamma}j_{2}\\
a_{2}e_{1}^{\beta}j_{1} & \mapsto U_{2}a_{1}e_{1}^{\beta}j_{1} +(U_{1}+U_{2}+U_{3}+U_{4})a_{2}e_{1}^{\gamma}j_{1}\\
a_{1}e_{1}^{\beta}j_{1} & \mapsto (U_{1}+U_{2}+U_{3}+U_{4})a_{1}e_{1}^{\gamma}j_{1}\\
a_{2}e_{1}^{\gamma}j_{2} & \mapsto U_{2}a_{1}e_{1}^{\gamma}j_{2} + U_{4}a_{2}e_{1}^{\gamma}j_{1}\\
a_{1}e_{1}^{\gamma}j_{2} & \mapsto U_{4}a_{1}e_{1}^{\gamma}j_{1}\\
a_{2}e_{1}^{\gamma}j_{1} & \mapsto U_{2}a_{1}e_{1}^{\gamma}j_{1} \\
a_{1}e_{1}^{\gamma}j_{1} & \mapsto 0  \\
\end{align*}

Canceling the isomorphisms $a_{i}e_{1}^{\alpha}j_{k} \mapsto a_{i}e_{2}j_{k}$ for $i,k=1,2$ gives us the complex

\begin{align*}
b_{2}g_{2}j_{2} & \mapsto U_{2}b_{1}g_{2}j_{2}+ U_{4}b_{2}g_{2}j_{1} + (U_{1}+U_{3})b_{2}g_{1}j_{2} +a_{2}e_{1}^{\beta}j_{2}\\
b_{1}g_{2}j_{2} & \mapsto U_{4}b_{1}g_{2}j_{1} + (U_{1}+U_{3})b_{1}g_{1}j_{2} + a_{2}e_{1}^{\gamma}j_{2}+ a_{1}e_{1}^{\beta}j_{2}\\
b_{2}g_{2}j_{1} & \mapsto U_{2}b_{1}g_{2}j_{1} + (U_{1}+U_{3})b_{2}g_{1}j_{1}+ a_{2}e_{1}^{\beta}j_{1} + a_{2}e_{1}^{\gamma}j_{2}\\
b_{1}g_{2}j_{1} & \mapsto (U_{1}+U_{3})b_{1}g_{1}j_{1} +a_{2}e_{1}^{\gamma}j_{1}+ a_{1}e_{1}^{\beta}j_{1} + a_{1}e_{1}^{\gamma}j_{2}\\ 
b_{2}g_{1}j_{2} & \mapsto U_{2}b_{1}g_{1}j_{2}+ U_{4}b_{2}g_{1}j_{1} + a_{2}e_{1}^{\gamma}j_{2}\\ 
b_{1}g_{1}j_{2} & \mapsto U_{4}b_{1}g_{1}j_{1} + a_{1}e_{1}^{\gamma}j_{2}\\
b_{2}g_{1}j_{1} & \mapsto U_{2}b_{1}g_{1}j_{1} + a_{2}e_{1}^{\gamma}j_{1}\\
b_{1}g_{1}j_{1} & \mapsto a_{1}e_{1}^{\gamma}j_{1} \\
a_{2}e_{1}^{\beta}j_{2} & \mapsto U_{2}a_{1}e_{1}^{\beta}j_{2} + U_{4}a_{2}e_{1}^{\beta}j_{1} +(U_{1}+U_{2}+U_{3}+U_{4})a_{2}e_{1}^{\gamma}j_{2}\\
a_{1}e_{1}^{\beta}j_{2} & \mapsto U_{4}a_{1}e_{1}^{\beta}j_{1}+(U_{1}+U_{2}+U_{3}+U_{4})a_{1}e_{1}^{\gamma}j_{2}\\
a_{2}e_{1}^{\beta}j_{1} & \mapsto U_{2}a_{1}e_{1}^{\beta}j_{1} +(U_{1}+U_{2}+U_{3}+U_{4})a_{2}e_{1}^{\gamma}j_{1}\\
a_{1}e_{1}^{\beta}j_{1} & \mapsto (U_{1}+U_{2}+U_{3}+U_{4})a_{1}e_{1}^{\gamma}j_{1}\\
a_{2}e_{1}^{\gamma}j_{2} & \mapsto U_{2}a_{1}e_{1}^{\gamma}j_{2} + U_{4}a_{2}e_{1}^{\gamma}j_{1}\\
a_{1}e_{1}^{\gamma}j_{2} & \mapsto U_{4}a_{1}e_{1}^{\gamma}j_{1}\\
a_{2}e_{1}^{\gamma}j_{1} & \mapsto U_{2}a_{1}e_{1}^{\gamma}j_{1} \\
a_{1}e_{1}^{\gamma}j_{1} & \mapsto 0  \\
\end{align*}

Both the subcomplex corresponding to the singularization and the quotient complex from the smoothing are isomorphic the the HOMFLY-PT homology of the graph with the cycle $Z_{1}$ removed. The edge map between them is an isomorphism, i.e. the homology corresponding to the cycle $Z_{1}$ in $D$ is trivial.

\newpage

\subsection{$Z_{2}$ - the Left Cycle}

For the cycle $Z_{2}$, we have the generators $(c, e, i)$ and $(c, d, h)$. The total complex is given by

\begin{align*}
c_{2}e_{1}^{\alpha}i_{2} & \mapsto U_{1}c_{1}e_{1}^{\alpha}i_{2}+U_{3}c_{2}e_{1}^{\alpha}i_{1} + c_{2}d_{1}h_{2}+ c_{2}e_{1}^{\beta}i_{2} + c_{2}e_{2}i_{2} \\
c_{1}e_{1}^{\alpha}i_{2} & \mapsto U_{3}c_{1}e_{1}^{\alpha}i_{1} + c_{1}d_{1}h_{2}+c_{1}e_{1}^{\beta}i_{2} +c_{1}e_{2}i_{2} +c_{2}e_{1}^{\gamma}i_{2}\\
c_{2}e_{1}^{\alpha}i_{1} & \mapsto U_{1}c_{1}e_{1}^{\alpha}i_{1} + c_{2}d_{1}h_{1}+c_{2}e_{1}^{\beta}i_{1} +c_{2}e_{2}i_{1}\\
c_{1}e_{1}^{\alpha}i_{1} & \mapsto c_{1}d_{1}h_{1}+c_{1}e_{1}^{\beta}i_{1} + c_{1}e_{2}i_{1} +c_{2}e_{1}^{\gamma}i_{1}\\
c_{2}d_{2}h_{2} & \mapsto U_{1}c_{1}d_{2}h_{2} + U_{3}c_{2}d_{2}h_{1} + U_{2}c_{2}d_{1}h_{2} +U_{4}c_{2}e_{2}i_{2}\\
c_{1}d_{2}h_{2} & \mapsto U_{3}c_{1}d_{2}h_{1} + U_{2}c_{1}d_{1}h_{2} +U_{4}c_{1}e_{2}i_{2}\\
c_{2}d_{2}h_{1} & \mapsto U_{1}c_{1}d_{2}h_{1} + U_{2}c_{2}d_{1}h_{1} +U_{4}c_{2}e_{2}i_{1} +U_{4}c_{2}e_{1}^{\gamma}i_{2}\\
c_{1}d_{2}h_{1} & \mapsto U_{2}c_{1}d_{1}h_{1} +U_{4}c_{1}e_{2}i_{1} +U_{4}c_{1}e_{1}^{\gamma}i_{2}\\
c_{2}d_{1}h_{2} & \mapsto U_{1}c_{1}d_{1}h_{2} + U_{3}c_{2}d_{1}h_{1} + U_{4}c_{2}e_{1}^{\gamma}i_{2}\\
c_{1}d_{1}h_{2} & \mapsto U_{3}c_{1}d_{1}h_{1} + U_{4}c_{1}e_{1}^{\gamma}i_{2}\\
c_{2}d_{1}h_{1} & \mapsto U_{1}c_{1}d_{1}h_{1}+ U_{4}c_{2}e_{1}^{\gamma}i_{1}\\
c_{1}d_{1}h_{1} & \mapsto U_{4}c_{1}e_{1}^{\gamma}i_{1}\\
c_{2}e_{2}i_{2} & \mapsto U_{1}c_{1}e_{2}i_{2}+U_{3}c_{2}e_{2}i_{1} +(U_{2}+U_{3})c_{2}e_{1}^{\gamma}i_{2}\\
c_{1}e_{2}i_{2} & \mapsto U_{3}c_{1}e_{2}i_{1}+(U_{2}+U_{3})c_{1}e_{1}^{\gamma}i_{2}\\
c_{2}e_{2}i_{1} & \mapsto U_{1}c_{1}e_{2}i_{1}+(U_{2}+U_{3})c_{2}e_{1}^{\gamma}i_{1}\\
c_{1}e_{2}i_{1} & \mapsto (U_{2}+U_{3})c_{1}e_{1}^{\gamma}i_{1}\\
c_{2}e_{1}^{\beta}i_{2} & \mapsto U_{1}c_{1}e_{1}^{\beta}i_{2}+U_{3}c_{2}e_{1}^{\beta}i_{1} +(U_{1}+U_{2}+U_{3}+U_{4})c_{2}e_{1}^{\gamma}i_{2}\\
c_{1}e_{1}^{\beta}i_{2} & \mapsto U_{3}c_{1}e_{1}^{\beta}i_{1}+(U_{1}+U_{2}+U_{3}+U_{4})c_{1}e_{1}^{\gamma}i_{2} \\
c_{2}e_{1}^{\beta}i_{1} & \mapsto U_{1}c_{1}e_{1}^{\beta}i_{1}+(U_{1}+U_{2}+U_{3}+U_{4})c_{2}e_{1}^{\gamma}i_{1} \\
c_{1}e_{1}^{\beta}i_{1} & \mapsto (U_{1}+U_{2}+U_{3}+U_{4})c_{1}e_{1}^{\gamma}i_{1}\\
c_{2}e_{1}^{\gamma}i_{2} & \mapsto U_{1}c_{1}e_{1}^{\gamma}i_{2}+U_{3}c_{2}e_{1}^{\gamma}i_{1} \\ 
c_{1}e_{1}^{\gamma}i_{2} & \mapsto U_{3}c_{1}e_{1}^{\gamma}i_{1} \\
c_{2}e_{1}^{\gamma}i_{1} & \mapsto U_{1}c_{1}e_{1}^{\gamma}i_{1} \\
c_{1}e_{1}^{\gamma}i_{1} & \mapsto 0 \\
\end{align*}

Reducing $c_{j}e_{1}^{\alpha}i_{k} \mapsto c_{j}e_{2}i_{k}$  for $j,k=1,2$ gives us the complex 

\begin{align*}
c_{2}d_{2}h_{2} & \mapsto U_{1}c_{1}d_{2}h_{2} + U_{3}c_{2}d_{2}h_{1} + (U_{2}+U_{4})c_{2}d_{1}h_{2} +U_{4}c_{2}e_{1}^{\beta}i_{2} \\
c_{1}d_{2}h_{2} & \mapsto U_{3}c_{1}d_{2}h_{1} + (U_{2}+U_{4})c_{1}d_{1}h_{2} +U_{4}(c_{1}e_{1}^{\beta}i_{2} +c_{2}e_{1}^{\gamma}i_{2})\\
c_{2}d_{2}h_{1} & \mapsto U_{1}c_{1}d_{2}h_{1} + (U_{2}+U_{4})c_{2}d_{1}h_{1} +U_{4}(c_{2}e_{1}^{\beta}i_{1} +U_{4}c_{2}e_{1}^{\gamma}i_{2})\\
c_{1}d_{2}h_{1} & \mapsto (U_{2}+U_{4})c_{1}d_{1}h_{1} +U_{4}(c_{1}e_{1}^{\beta}i_{1} +c_{2}e_{1}^{\gamma}i_{1}+c_{1}e_{1}^{\gamma}i_{2})\\
c_{2}d_{1}h_{2} & \mapsto U_{1}c_{1}d_{1}h_{2} + U_{3}c_{2}d_{1}h_{1} + U_{4}c_{2}e_{1}^{\gamma}i_{2}\\
c_{1}d_{1}h_{2} & \mapsto U_{3}c_{1}d_{1}h_{1} + U_{4}c_{1}e_{1}^{\gamma}i_{2}\\
c_{2}d_{1}h_{1} & \mapsto U_{1}c_{1}d_{1}h_{1}+ U_{4}c_{2}e_{1}^{\gamma}i_{1}\\
c_{1}d_{1}h_{1} & \mapsto U_{4}c_{1}e_{1}^{\gamma}i_{1}\\
c_{2}e_{1}^{\beta}i_{2} & \mapsto U_{1}c_{1}e_{1}^{\beta}i_{2}+U_{3}c_{2}e_{1}^{\beta}i_{1} +(U_{1}+U_{2}+U_{3}+U_{4})c_{2}e_{1}^{\gamma}i_{2}\\
c_{1}e_{1}^{\beta}i_{2} & \mapsto U_{3}c_{1}e_{1}^{\beta}i_{1}+(U_{1}+U_{2}+U_{3}+U_{4})c_{1}e_{1}^{\gamma}i_{2} \\
c_{2}e_{1}^{\beta}i_{1} & \mapsto U_{1}c_{1}e_{1}^{\beta}i_{1}+(U_{1}+U_{2}+U_{3}+U_{4})c_{2}e_{1}^{\gamma}i_{1} \\
c_{1}e_{1}^{\beta}i_{1} & \mapsto (U_{1}+U_{2}+U_{3}+U_{4})c_{1}e_{1}^{\gamma}i_{1}\\
c_{2}e_{1}^{\gamma}i_{2} & \mapsto U_{1}c_{1}e_{1}^{\gamma}i_{2}+U_{3}c_{2}e_{1}^{\gamma}i_{1} \\ 
c_{1}e_{1}^{\gamma}i_{2} & \mapsto U_{3}c_{1}e_{1}^{\gamma}i_{1} \\
c_{2}e_{1}^{\gamma}i_{1} & \mapsto U_{1}c_{1}e_{1}^{\gamma}i_{1} \\
c_{1}e_{1}^{\gamma}i_{1} & \mapsto 0 \\
\end{align*}

As in the case of $Z_{1}$, the subcomplex and quotient complex corresponding to the singularization and smoothing, respectively, are both isomorphic to the HOMFLY-PT complex of the graph minus our cycle. However, instead of the edge map being the canonical isomorphism, it is given by multiplication by $U_{4}$.

\subsection{$Z_{3}$ - The First Diagonal}

Unlike the previous three calculations, the cycle $Z_{3}$ only appears in one of the resolutions, in this case the singularization. The complex of the oriented smoothing, given by 

\[X^{\alpha} \xrightarrow{\Phi_{A^{-}}}Y\]

\noindent
comes from the Heegaard diagram in Figure \ref{HDCrossing}, with $X$'s placed at the $B$'s. This diagram can be changed to the standard diagram for the smoothing via isotopies and handleslides that take place away from the additional basepoints. Since this new diagram does not contain any generators with the connectivity of $Z_{3}$, it follows that the original complex for the smoothing was acyclic.

Thus, we are left with the subcomplex corresponding to the singularization

\[X^{\beta} \xrightarrow{U_{1}+U_{2}+U_{3}+U_{4}}X^{\gamma} \]

\noindent
$X$ has generators $a_{j}e_{1}i_{k}$ for $j,k=1,2$, so for each of the $X$'s we get the Koszul complex shown below.

\begin{figure}[!h]
\centering
\begin{tikzpicture}
  \matrix (m) [matrix of math nodes,row sep=5em,column sep=5em,minimum width=2em] {
    a_{2}e_{1}i_{2} &a_{1}e_{1}i_{2} \\
    a_{2}e_{1}i_{1} & a_{1}e_{1}i_{1} \\};
  \path[-stealth]
    (m-1-1) edge node [left] {$U_{3}$} (m-2-1)
            edge node [above] {$U_{2}$} (m-1-2)
    (m-2-1) edge node [below] {$U_{2}$} (m-2-2)
    (m-1-2) edge node [right] {$U_{3}$} (m-2-2);
\end{tikzpicture}
\end{figure}

\noindent
Hence, the total complex is a cube complex generated by the three edge maps $U_{2}$, $U_{3}$, and $U_{1}+U_{2}+U_{3}+U_{4}$. Canceling the $U_{2}$ and $U_{3}$ maps, we get the HOMFLY-PT complex of the graph with $Z_{3}$ removed.

\subsection{$Z_{4}$ - The Second Diagonal}

We can apply the same arguments used in the previous section to show that \[X^{\alpha} \xrightarrow{\Phi_{A^{-}}}Y\] is acyclic, or we can see directly that $\Phi_{A^{-}}$ is actually an isomorphism. Either way, it is apparent that the smoothing has trivial homology, so we are once again left with 

\[X^{\beta} \xrightarrow{U_{1}+U_{2}+U_{3}+U_{4}}X^{\gamma} \]

\noindent
$X$ has generators $c_{i}e_{1}j_{k}$ for $i,k=1,2$, so for each of the $X$'s we get the complex

\begin{figure}[!h]
\centering
\begin{tikzpicture}
  \matrix (m) [matrix of math nodes,row sep=5em,column sep=5em,minimum width=2em] {
    c_{2}e_{1}j_{2} &c_{1}e_{1}j_{2} \\
    c_{2}e_{1}j_{1} & c_{1}e_{1}j_{1} \\};
  \path[-stealth]
    (m-1-1) edge node [left] {$U_{4}$} (m-2-1)
            edge node [above] {$U_{1}$} (m-1-2)
    (m-2-1) edge node [below] {$U_{1}$} (m-2-2)
    (m-1-2) edge node [right] {$U_{4}$} (m-2-2);
\end{tikzpicture}
\end{figure}

Hence, the total complex is a cube complex generated by the three edge maps $U_{1}$, $U_{4}$, and $U_{1}+U_{2}+U_{3}+U_{4}$. Canceling the $U_{1}$ and $U_{4}$ maps, we get the HOMFLY-PT complex of the graph with $Z_{4}$ removed.

\subsection{The Full Cycle $Z_{5}$}

Like the previous two cycles, $Z_{5}$ does not appear as a cycle in both the smoothing and the singularization - in this case it only appears in the smoothing. This is readily apparent from the fact that $e$ is not included in any of the generators, hence $X$ is trivial. Thus, our total complex is $Y$, which is the Koszul complex on $U_{1},U_{2},U_{3}$, and $U_{4}$. Canceling all four of these once again gives the HOMFLY-PT complex of the diagram with the cycle $Z_{5}$ removed.

\bibliography{TriplyGradedHomology}{}
\bibliographystyle{plain}

\end{document}